\documentclass[smallextended]{springer/svjour3}       
\smartqed  
\usepackage[utf8]{inputenc}
\usepackage[tbtags]{amsmath}
\usepackage[subnum]{cases}
\usepackage{amssymb,bm}

\usepackage{amsthm,hyperref}
\usepackage{nomencl}
\usepackage{enumitem}
\usepackage{graphicx,wasysym,color,subfigure,lscape, threeparttable, textgreek}

\usepackage{pdflscape}
\usepackage{rotating}

\usepackage{etoolbox}
\usepackage{multirow}
\usepackage{booktabs}
\usepackage{rotating}
\usepackage{algorithm}
\usepackage{appendix}
\usepackage{algpseudocode}%
\usepackage{wrapfig}
\usepackage{tikz}
\newcommand{\Mypm}{\mathbin{\tikz [x=1.4ex,y=1.4ex,line width=.1ex] \draw (0.0,0) -- (1.0,0) (0.5,0.08) -- (0.5,0.92) (0.0,0.5) -- (1.0,0.5);}}%

\newcommand{\kaarthik}[1]{\textcolor{black}{#1}}
\newcommand{\harsha}[1]{\textcolor{black}{#1}}

\algnewcommand\algorithmicinput{\textbf{Input:}}
\algnewcommand\Input{\item[\algorithmicinput]}

\algnewcommand\algorithmicoutput{\textbf{Output:}}
\algnewcommand\Output{\item[\algorithmicoutput]}

\begin{document}

\title{An Adaptive, Multivariate Partitioning Algorithm for Global Optimization of Nonconvex Programs}


\author{}
 \author{Harsha Nagarajan \and Mowen Lu \and Site Wang \and Russell Bent \and Kaarthik Sundar}

 \institute{Harsha Nagarajan \and Russell Bent \and Kaarthik Sundar\at
                Center for Nonlinear Studies, Los Alamos National Laboratory \\
               \email{\{harsha, rbent, kaarthik\}@lanl.gov}           
           \and
           Mowen Lu \and Site Wang \at
           Department of Industrial Engineering,
           Clemson University \\
           \email{\{mlu87, sitew\}@g.clemson.edu}
 }

\date{Received: date / Accepted: date}

\maketitle

\begin{abstract}
\kaarthik{In this work, we develop an adaptive, multivariate partitioning algorithm for solving nonconvex, Mixed-Integer Nonlinear Programs (MINLPs) with polynomial functions to global optimality. In particular, we present an iterative algorithm that exploits piecewise, convex relaxation approaches via disjunctive formulations to solve MINLPs that is different than conventional spatial branch-and-bound approaches. The algorithm partitions the domains of variables in an adaptive and non-uniform manner at every iteration to focus on productive areas of the search space. Furthermore, domain reduction techniques based on sequential, optimization-based bound-tightening and piecewise relaxation techniques, as a part of a presolve step, are integrated into the main algorithm. Finally, we demonstrate the effectiveness of the algorithm on well-known benchmark problems (including Pooling and Blending instances) from MINLPLib and compare our algorithm with state-of-the-art global optimization solvers. With our novel approach, we solve several large-scale instances, some of which are not solvable by state-of-the-art solvers. We also succeed in reducing the best known optimality gap for a hard, generalized pooling problem instance.}

\end{abstract}

\section{Introduction}
Mixed-Integer Nonlinear Programs (MINLPs) are convex/non-convex, mathematical programs that include discrete variables and nonlinear terms in the objective function and/or constraints. 
In practice, non-convex MINLPs arise in many applications such as chemical engineering (synthesis of process and water networks) \cite{meyer2006global,ryoo1995global}, energy infrastructure networks \cite{hijazi2016,lu_optimal_tps2017,nagarajan2016optimal}, and in molecular distance geometry problems \cite{liberti2008branch}, to name a few. 
\kaarthik{
Given the importance of these problems, considerable research has been devoted to developing approaches for solving MINLPs, such as approaches implemented in such solvers as
BARON \cite{sahinidis1996baron}, Couenne \cite{belotti2009branching} and SCIP \cite{achterberg2009scip}. Within these approaches, two of the key features of successful methods include MINLP relaxations and search. For example, in a typical solver, 
non-convex terms are replaced with convex over- and under-estimators \cite{boukouvala2016global}. The resulting convex optimization problem is a relaxation of the original MINLP and its solution is a bound to the optimal objective value of the MINLP. These relaxations are then used in conjunction with a search procedure, like spatial branch-and-bound (sBB), to explore the solution space of the MINLP and identify the global optimal solution.}

Despite major developments related to these features and others, MINLPs still remain difficult to solve and global optimization solvers often struggle to find optimal solutions and at times, even a feasible solution. In many cases, the source of these struggles are weak relaxations of the MINLP \kaarthik{and the impact weak relaxations on the size of the search space that is explored. 
To address these difficulties, in this paper we develop an approach for deriving better bounds through piecewise convex relaxations that are modeled as mixed-integer convex optimization problems.
The piecewise convex relaxations are combined with additional algorithmic enhancements, a novel, adaptive domain partitioning scheme, and successive solves of mixed-integer problems (MIP), to produce a novel search procedure. This global optimization algorithm 
is tested extensively on MINLPs with polynomial constraints, including the well-known and hard Pooling and Blending instances \cite{kolodziej2013discretization} and is compared with state-of-the-art global optimization approaches.  In this paper, we focus on MINLPs with polynomial constraints, but the approach is fairly generic and can be generalized to other nonconvex functions. Finally, for ease of exposition, we assume the MINLPs are minimization problems throughout the rest of the paper.
}

\kaarthik{ 
We next discuss the key contributions we make in this paper. Our first contribution improves convex relaxations of polynomial functions. Here, we develop an approach based on piecewise convex relaxations. While such relaxations have been used to
bound medium-sized MINLPs with bilinear functions 
\cite{castro2015tightening,hasan2010piecewise,karuppiah2006global,kolodziej2013discretization,bergamini2008improved}, we generalize these approaches to arbitrary polynomial functions.} 

\kaarthik{ 
Our second contribution turns the derivation of these relaxations into a search procedure based on solving MIPs, i.e., a `MIP-based approach' that is akin to the approach of \cite{ruiz2017global,faria2012new}. 
%
Most existing approaches rely on sBB. In a conventional sBB algorithm, 
branching occurs on the domain of one variable at a time. The branching generates two new problems (child nodes), each with a smaller domain than the parent problem (node) and potentially tighter relaxations. Whenever the best possible solution at a node is worse than the best known feasible solution, the node is pruned. sBB is typically combined with 
with enhancements such as cutting planes and domain reduction techniques to further improve the efficiency of the search \cite{sahinidis1996baron,tawarmalani2005polyhedral}.
In contrast, our MIP-based approach solves a sequence of MIPs based on successively tighter piecewise convex relaxations that converges to the optimal solution.
} 

\kaarthik{
Our third contribution is a sparse domain partitioning approach for piecewise convex relaxations. Most existing partitioning approaches rely on uniform partitioning, i.e., \cite{hasan2010piecewise}.
Unfortunately, uniform partitioning, when used in conjunction with a MIP-based approach can lead to MIPs with a large number of binary variables.
Thus, uniform partitioning limits MIP based approaches
to small- and medium-sized problems. This important issue has motivated the development of piecewise relaxation techniques where the number of binary variables increases logarithmically \cite{misener2011apogee,Vielma2009} with the number of partitions and multiparametric disaggregation approaches \cite{castro2015normalized}. 
In other work \cite{wicaksono2008piecewise}, the authors present a non-uniform, bivariate partitioning approach that improves the relaxations but provide results for a single, simple benchmark problem. Reference \cite{teles2013univariate} discusses a univariate parametrization method applied to medium-sized benchmarks. However, none of these approaches address the key limitation of uniform partitioning, \emph{partition density}, i.e. these methods introduce partitions in unproductive areas of the variable domains. We address this limitation by introducing a novel approach that adaptively  partitions the relaxations in regions of the search space that favor optimality. To the best of our knowledge, this is the first work in the literature that develops a complete MIP-based method for solving MINLPs to global optimality based on sparse domain partitioning schemes.}

\kaarthik{Our fourth (minor) contribution combines the adaptively partitioned piecewise relaxation approach with sequential, optimization-based bound-tightening (OBBT). OBBT is used as a presolve step in the overall global optimization algorithm. 
OBBT solves a sequence of convex minimization and maximization problems on the variables that appear in nonconvex terms. The solutions to these problems tighten domains of the variables and the associated relaxation to the nonconvex terms \cite{puranik2017domain,belotti2012feasibility,faria2011novel,mouret2009tightening}. 
Recent work has observed the effectiveness of applying OBBT in various applications \cite{coffrin2015strengthening,fei2017acc,nagarajan2017r2r}. We adapt and extend this approach by solving \textit{convex MIPs} in the OBBT procedure (existing approaches solve ordinary convex problems). Though this approach seems counter-intuitive, computational experiments indicate that the value of the strengthened bounds obtained by solving MIPs often outweigh the computational time required to solve them.}

\kaarthik{Finally, these four contributions are combined into a MIP-based global optimization algorithm which is referred to as the Adaptive, Multivariate Partitioning (AMP) algorithm. Given an MINLP, AMP first calculates a local solution to the MINLP, an initial lower bound, and tightened variable bounds (sequential OBBT) as a presolve step. 
The main loop of the AMP algorithm refines the partitions of the variable domain, computes improved lower bounds, and derives better local (upper bound) solutions. The variable domains are refined in a non-uniform and adaptive fashion. In particular, partitions are dynamically added around the optimal solution to the relaxed problem at each iteration of AMP. This loop iterates until the relative gap between the lower bound and the upper bound solution meets a user specified global optimality tolerance. The computation may also be interrupted early to provide a local optimal solution.}

A preliminary version of this work \cite{nagarajan2016tightening} was  applied to hard, infrastructure network optimization problems  \cite{fei2017acc,LuPSCC2017}, which demonstrated the effectiveness of adaptive partitioning strategies. Given the efficacy of the proposed ideas, including various enhancements (not discussed in this paper), the algorithm's implementation is also available as an open-source solver in Julia programming language \cite{bent2017polyhedral}. 
The remainder of this paper is organized as follows: Section \ref{sec:problem} discusses the required notation, problem set-up, and reviews standard convex relaxations. Section \ref{sec:algo} discusses our Adaptive, Multivariate Partitioning Algorithm to solve MINLPs to global optimality with a few proofs of convergence guarantees. Section \ref{sec:results} illustrates the strength of the algorithms on benchmark MINLPs and Section \ref{sec:conclusions} concludes the paper.

\section{Definitions}
\label{sec:problem}
\newcommand{\calP}{\boldsymbol{\mathcal{P}}}
\newcommand{\calI}{\boldsymbol{\mathcal{I}}}

\paragraph{Notation} Here, we use lower and upper case for vector and matrix entries, respectively. Bold font refers to the entire vector or matrix.
With this notation, $||\bm{v}||_{\infty}$ defines the $\ell^{\infty}$ norm of vector $\bm{v}\in\mathbb{R}^n$. Given vectors $\bm{v}_1 \in\mathbb{R}^n$ and $\bm{v}_2 \in\mathbb{R}^n$, $\bm{v}_1 \cdot \bm{v}_2 = \sum_{i=1}^n {v_1}_i {v_2}_i$; $\bm{v}_1 + \bm{v}_2$ implies element-wise sums; and $\frac{\bm{v}_1}{\alpha}$ denotes the element-wise ratio between entries of $\bm{v}_1$ and the scalar $\alpha$. Next, $z \in \mathbb{Z}$ represents an integer (variable/constant) and specifically $z \in \mathbb{B}$ represents a binary variable. Finally, we let $\bm e_i$ denote a unit vector whose $i$\textsuperscript{th} coordinate is one. 

\paragraph{Problem} The problems considered in this paper are MINLPs with polynomials which have at least one feasible solution. 
The general form of the problem, denoted as $\calP$, is as follows: 

\begin{equation*}
\begin{aligned}
\calP: \ \ \ & \underset{\bm{x},\bm{y}}{\text{minimize}} & &  f(\bm{x},\bm{y}) \\
&\text{subject to} & & \bm{g}(\bm{x},\bm{y}) \leqslant 0, \\
& & & \bm{h}(\bm{x},\bm{y}) = 0, \\
& & & \bm{x}^L \leqslant \bm{x} \leqslant \bm{x}^U, \\
& & & \bm{y} \in \{0,1\}^m
\end{aligned}
\end{equation*}

\noindent where, $f:\mathbb{R}^n\times \mathbb{B}^m \rightarrow \mathbb{R}$, $g_i :\mathbb{R}^n\times \mathbb{B}^m \rightarrow \mathbb{R} \ \mathrm{for} \ i=1,\ldots,G$ and $h_i :\mathbb{R}^n\times \mathbb{B}^m \rightarrow \mathbb{R}  \ \mathrm{for} \  i=1,\ldots,H$ \kaarthik{are polynomials.} 
For the sake of clarity, neglecting the binary variables in the functions, $f,\bm{g}$ or $\bm{h}$ can assume the following form: 
\begin{align}
    \sum_{t \in T} a_t \prod_{k\in K_t}x^{\alpha_k}_k
    \label{eq:ML_form}
\end{align}

\noindent
where, \kaarthik{$T$ is a set of terms in a polynomial, $K_t$ is a set of variables in term $t$,} $a_t \in \mathbb{R}$ is a \kaarthik{real} coefficient and $\alpha_k$ is an exponent \kaarthik{(integer)} value. 
$\bm{x}$ and $\bm{y}$ are vectors of continuous variables with box constraints [$\bm{x}^L,\bm{x}^U$] and binary variables, respectively. 
$\bm{x}$ and $\bm{y}$ have dimension $n$ and $m$, respectively.
We use notation $\sigma$ to denote a solution to $\calP$, where $\sigma(\cdot)$ is the value of variable(s), $\cdot$, in $\sigma$ and $f(\sigma)$ is the objective value of $\sigma$. We note that $\calP$ is an NP-hard combinatorial problem. \kaarthik{The construction of convex relaxations for each individual term in Eq. \eqref{eq:ML_form} plays a critical role in developing algorithms for solving $\calP$ to global optimality. In the following paragraphs, we discuss the relaxations used in this paper.}

\harsha{In this paper, we use relaxations for bilinear, multilinear and quadratic monomials. Note that, without loss of generality, any polynomial can be equivalently expressed using a combination of these monomials.}

\paragraph{McCormick relaxation of a bilinear term} \kaarthik{For $t\in T$, when $|K_t| \leqslant 2$
and $\alpha_k = 1$, the McCormick relaxation \cite{mccormick1976computability} is used. Given variables $x_i$ and $x_j$ that appear in $t$, McCormick relaxed the set
$$S_B  = \left\{(x_i,x_j,\widehat{x_{ij}}) \in [x_i^L,x_i^U]\times [x_j^L,x_j^U] \times \mathbb{R} \mid \widehat{x_{ij}} = x_ix_j\right\}$$ with the following four inequalities:}

{\fontsize{9}{8}\selectfont
\begin{subequations}
\begin{align}
   \widehat{x_{ij}} &\geqslant x_i^L x_j + x_j^L x_i - x_i^Lx_j^L\\ 
   \widehat{x_{ij}} &\geqslant x_i^U x_j + x_j^U x_i - x_i^Ux_j^U\\ 
   \widehat{x_{ij}} &\leqslant x_i^L x_j + x_j^U x_i - x_i^Lx_j^U\\ 
   \widehat{x_{ij}} &\leqslant x_i^U x_j + x_j^L x_i - x_i^Ux_j^L 
\end{align}
\label{eq:mcc}
\end{subequations}}

\noindent
Let $\langle x_i,x_j\rangle^{MC}\supset S_B$ represent the feasible region defined by \eqref{eq:mcc}. For a single bilinear term $x_ix_j$, the relaxations in \eqref{eq:mcc} describe the convex hull of set $S_B$ \cite{al1983jointly}.

\paragraph{Recursive McCormick relaxation of a multilinear term}
For a general multilinear term ($|K_t| \geqslant 3$, $\alpha_k = 1$), McCormick proposed a recursive approach to successively derive envelopes on bilinear combinations of the terms. The resulting relaxation has formed the basis for the relaxations used in the global optimization literature, including the implementations in BARON, Couenne and SCIP \cite{sahinidis1996baron,belotti2009branching,achterberg2009scip}. More formally, the non-convex function given by $\prod_{k=1}^{|K_t|} x_k$ can be relaxed by introducing lifted variables $\widehat{x}_1,\ldots,\widehat{x}_{|K_t|-1}$ such that $\widehat{x}_1 = x_{1}x_{2}$ and $\widehat{x}_{i} = \widehat{x}_{i-1}x_{i+1}$ for every $i=2,\ldots,|K_t|-1$. Thus, the recursive McCormick envelopes of $\prod_{k=1}^{|K_t|} x_k$ are described by 
\begin{subequations}
\begin{align}
& \left\{(x_1,x_2,\widehat{x}_1) \in [x_1^L,x_1^U]\times [x_2^L,x_2^U] \times [\widehat{x}_1^L,\widehat{x}_1^U] \mid \widehat{x}_1 = \langle x_1,x_2\rangle^{MC}\right\}, \\
& \left\{(\widehat{x}_{i-1},x_{i+1},\widehat{x}_i) \in [\widehat{x}_{i-1}^L,\widehat{x}_{i-1}^U]\times [x_{i+1}^L,x_{i+1}^U] \times [\widehat{x}_i^L,\widehat{x}_i^U]
\mid  \right. \nonumber \\  
& \hspace{3.4cm} \left. \widehat{x}_i = \langle \widehat{x}_{i-1},x_{i+1} \rangle^{MC}\right\}, \quad \forall i=2,\ldots,|K_t|-1.
\end{align}
\label{eq:rec_mcc}
\end{subequations}
\noindent
where, the bounds of $\widehat{x}_i$ variables are derived appropriately. By abuse of notation, \eqref{eq:rec_mcc} can be succinctly represented as
$$\left\langle \prod_{k=1}^{|K_t|} x_k\right\rangle^{MC} = \left\langle\left\langle\left\langle x_1,x_2\right\rangle^{MC},\ldots, x_{|K_t|-1}\right\rangle^{MC},x_{|K_t|} \right\rangle^{MC}.$$ 

\kaarthik{In general, the recursive McCormick envelopes described in \eqref{eq:rec_mcc} for a single multilinear term are not the tightest possible relaxation. The choice of the recursion order affects the tightness of the relaxation \cite{Speakman2017,cafieri2010convex}. However, authors in \cite{ryoo2001analysis} prove that \eqref{eq:rec_mcc} describes the convex hull when the bounds on the variables are in the set $[0,1]$. This result was generalized by \cite{luedtke2012some} for variables with bounds that are either $[0,x^U_i]$ or [$-x^U_i,x^U_i$] (symmetric about the origin). More generally, the convex hull of a multilinear term can be obtained by using an extreme point characterization by using exponential number of variables \cite{Rikun1997}. The computational tractability of using an extreme point characterization for piecewise relaxation of multilinear terms remains a subject of future work.}

\paragraph{Piecewise McCormick relaxation of a bilinear term}
\kaarthik{In the presence of partitions on the variables involved in a multilinear term, the McCormick relaxations (applied on bilinear terms) can be tightened (see Figure \ref{fig:mcc_region}[a]) by using a piecewise convex relaxation which uses one binary variable per variable partition.} Given a bilinear term $x_ix_j$ and partition sets $\calI_i$ and $\calI_j$, binary variables
$\widehat{\bm{y}}_i \in \{0,1\}^{|\calI_i|}$ and $\widehat{\bm{y}}_j \in \{0,1\}^{|\calI_j|}$ 
are used to denote these partitions. Each entry {\color{blue} in} $\calI_i$ is a pair of values, $\langle i,j\rangle$ that model the upper and lower bound of a variable in a partition. We refer to the collection of all partition sets with $\calI$.
These binary variables are used to control the partitions that are active and the associated relaxation of the active partition.
Formally, the piecewise McCormick constraints, 
denoted by $\widehat{x_{ij}} \in \langle x_i,x_j\rangle^{MC(\calI)}$,
take the following form: 
\begin{subequations}
\begin{align}
   &\widehat{x_{ij}} \geqslant (\bm{x}_i^l\cdot\widehat{\bm{y}}_i) x_j + (\bm{x}_j^l\cdot\widehat{\bm{y}}_j) x_i - (\bm{x}_i^l\cdot\widehat{\bm{y}}_i)(\bm{x}_j^l\cdot\widehat{\bm{y}}_j) \label{eq:mcc1}\\ 
   &\widehat{x_{ij}} \geqslant (\bm{x}_i^u\cdot\widehat{\bm{y}}_i) x_j + (\bm{x}_j^u\cdot\widehat{\bm{y}}_j) x_i - (\bm{x}_i^u\cdot\widehat{\bm{y}}_i)(\bm{x}_j^u\cdot\widehat{\bm{y}}_j)\label{eq:mcc2}\\ 
   &\widehat{x_{ij}} \leqslant (\bm{x}_i^l\cdot\widehat{\bm{y}}_i) x_j + (\bm{x}_j^u\cdot\widehat{\bm{y}}_j) x_i - (\bm{x}_i^l\cdot\widehat{\bm{y}}_i)(\bm{x}_j^u\cdot\widehat{\bm{y}}_j) \label{eq:mcc3}\\ 
   &\widehat{x_{ij}} \leqslant (\bm{x}_i^u\cdot\widehat{\bm{y}}_i) x_j + (\bm{x}_j^l\cdot\widehat{\bm{y}}_j) x_i - (\bm{x}_i^u\cdot\widehat{\bm{y}}_i)(\bm{x}_j^l\cdot\widehat{\bm{y}}_j) \label{eq:mcc4}\\
   &\widehat{\bm y}_i \cdot \bm 1 = 1, \ \ \widehat{\bm y}_j \cdot \bm 1 = 1 \\
   &\widehat{\bm{y}}_i \in \{0,1\}^{|\calI_i|}, \quad \widehat{\bm{y}}_j \in \{0,1\}^{|\calI_j|}
\end{align}
\label{eq:tmc}
\end{subequations}

\noindent where, $(\bm{x}_i^l, \bm{x}_i^u) \in \calI_i$ are the vector form of the partition sets of $x_i$ ($\calI_i$) and $\bm 1$ is a vector of ones of appropriate dimension.
Also, $(\bm{x}_i^l\cdot\widehat{\bm{y}}_i)(\bm{x}_j^l\cdot\widehat{\bm{y}}_j)$ is rewritten as  $\bm{x}_i^l(\widehat{\bm{y}}_i\widehat{\bm{y}}_j^T)\bm{x}_j^l$, where $(\widehat{\bm{y}}_i\widehat{\bm{y}}_j^T)$ is a matrix with binary product entries. Note that these binary products and the bilinear terms in $\widehat{\bm{y}}_j x_i$ and $\widehat{\bm{y}}_i x_j$ can be linearized exactly using standard McCormick relaxations \cite{nagarajan2018lego}. 

It is then straightforward to generalize piecewise McCormick relaxations to multilinear terms, and we use the following notation to denote these relaxations

\[
\left\langle \prod_{k=1}^{|K_t|} x_k\right\rangle^{MC(\calI)} = \left\langle\left\langle\left\langle x_1,x_2\right\rangle^{MC(\calI)},\ldots, x_{|K_t|-1}\right\rangle^{MC(\calI)}, x_{|K_t|} \right\rangle^{MC(\calI)}.
\] 

\noindent We also note that the McCormick relaxation is a special case of the piecewise McCormick relaxation when $\calI_i = \{\langle x^L_i,x^U_i\rangle\}$. 

These relaxations can also be encoded using $\log(|\calI_i|)$ binary variables \cite{DAmbrosio2010,Li2013b,Vielma2009} or variations of special order sets (SOS1, SOS2). 
For an ease of exposition, we do not present the details of the log-based formulation in this paper. However, later in the results section, we do compare the effectiveness of SOS1 formulations with respect to the linear representation of binary variables. 

\begin{figure}[htp]
   \centering
   \subfigure[Bilinear term $(x_ix_j)$]{
   \includegraphics[scale=0.305]{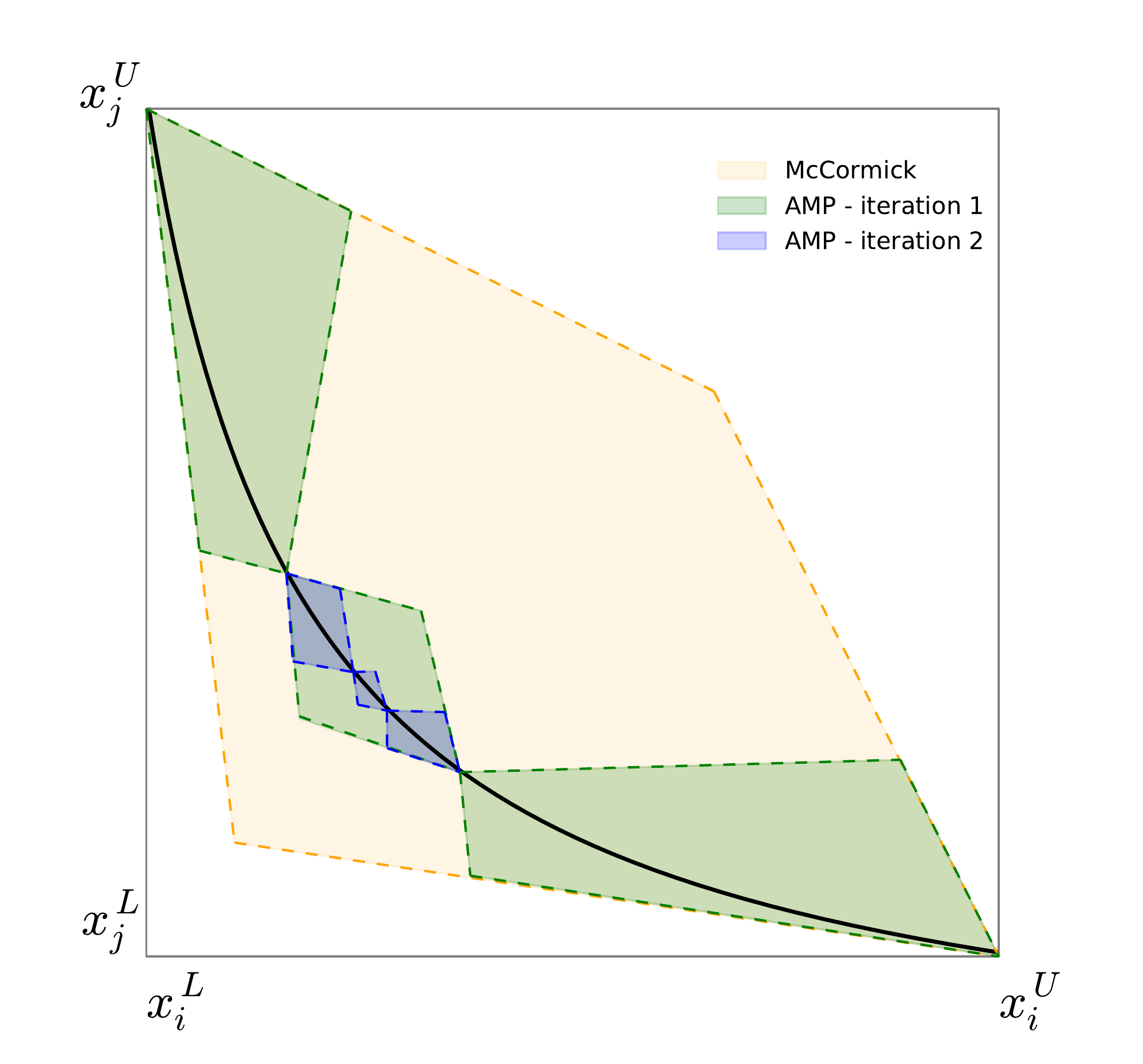}}
   \subfigure[Quadratic term $(x_i^2)$]{
   \includegraphics[scale=0.9]{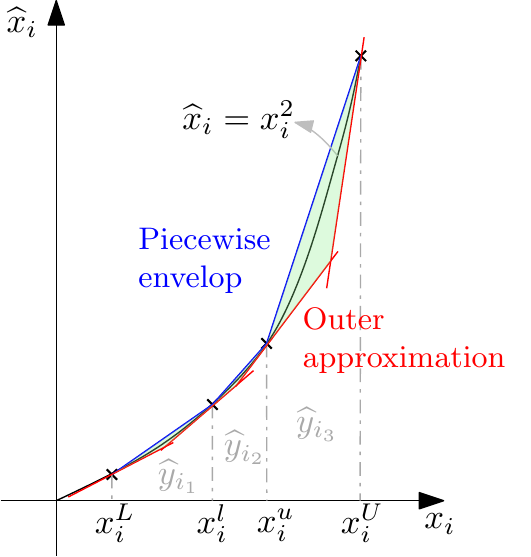}} 
   \caption{Piecewise relaxations (shaded) of bilinear and quadratic terms for a given set of partitions.}
   \label{fig:mcc_region}
\end{figure}

\paragraph{Piecewise relaxation of a quadratic term}
\harsha{Without loss of generality\footnote{In the case of a higher order univariate monomial, i.e., $x_i^5$,  apply a reduction of the form $x_i^2x_i^2x_i \Rightarrow \tilde{x}_i^2x_i \Rightarrow \tilde{\tilde{x}}_ix_i$.}, assume a univariate monomial takes the form $x_i^2$. Though, we restrict our discussion to a univariate monomial, similar extensions hold true for a multivariate monomial by applying  a sequence of relaxations on the respective univariate monomials}. Given partitions in $\calI_i$,
the piecewise, convex relaxation (see Figure \ref{fig:mcc_region}[b]), denoted by $\widehat{x}_i\in \langle x_i\rangle^{MC_q(\calI)}$, takes the form: 
\begin{subequations}
\begin{align}
   &\widehat{x}_i \geqslant x_i^2, \label{eq:quad}\\ 
   &\widehat{x}_i \leqslant  \left((\bm{x}_i^l\cdot\widehat{\bm{y}}_i) + (\bm{x}_i^u\cdot\widehat{\bm{y}}_i)\right) x_i - (\bm{x}_i^l\cdot\widehat{\bm{y}}_i)(\bm{x}_i^u\cdot\widehat{\bm{y}}_i)\\ 
   &\widehat{\bm y}_i \cdot \bm 1 = 1 \\
   &\widehat{\bm{y}}_i \in \{0,1\}^{|\calI_i|}
\end{align}
\label{eq:qc}
\end{subequations}

\noindent
Once again, $(\bm{x}_i^l\cdot\widehat{\bm{y}}_i)(\bm{x}_i^u\cdot\widehat{\bm{y}}_i)$ is rewritten as  $\bm{x}_i^l(\widehat{\bm{y}}_i\widehat{\bm{y}}_i^T)\bm{x}_i^u$, where $\widehat{\bm{y}}_i\widehat{\bm{y}}_i^T$ is a symmetric matrix with binary product entries (squared binaries on diagonal). Hence, it is sufficient to linearize the entries of the upper triangular matrix with exact representations. We also note again that the unpartitioned relaxation is a special case where $\calI_i = \{ \langle x^L_i,x^U_i \rangle\}$.

\begin{lemma}
\kaarthik{$\langle x_i\rangle^{MC_q(\calI)} \subset \langle x_i,x_i\rangle^{MC(\calI)} $.}
\end{lemma}
\begin{proof}
\kaarthik{
Given $\calI_i$ for variable $x_i$, $\langle x_i,x_i\rangle^{MC(\calI)}$ is given by the following constraints:
\begin{subequations}
\begin{align}
\label{eq:pf1}
    &\widehat{x}_i \geqslant  2(\bm{x}_i^l\cdot\widehat{\bm{y}}_i) x_i - (\bm{x}_i^l\cdot\widehat{\bm{y}}_i)^2\\
    \label{eq:pf2}
    &\widehat{x}_i \geqslant  2(\bm{x}_i^u\cdot\widehat{\bm{y}}_i) x_i - (\bm{x}_i^u\cdot\widehat{\bm{y}}_i)^2\\ \label{eq:pf3}
   &\widehat{x_i} \leqslant  \left((\bm{x}_i^l\cdot\widehat{\bm{y}}_i) + (\bm{x}_i^u\cdot\widehat{\bm{y}}_i)\right) x_i - (\bm{x}_i^l\cdot\widehat{\bm{y}}_i)(\bm{x}_i^u\cdot\widehat{\bm{y}}_i) \\ 
   &\widehat{\bm y}_i \cdot \bm 1 = 1, \, \widehat{\bm{y}}_i \in \{0,1\}^{|\calI_i|}
\end{align}
\label{eq:dtmc_proof}
\end{subequations}
\noindent
First, we claim that any point in $\langle x_i\rangle^{MC_q(\calI)}$ also lies in $\langle x_i,x_i\rangle^{MC(\calI)}$. 
This is trivial to observe since Eqs. \eqref{eq:pf1} and \eqref{eq:pf2} are outer approximations of Eq. \eqref{eq:quad} at the partition points. 
To prove $\langle x_i\rangle^{MC_q(\calI)}$ is a strict subset of $\langle x_i,x_i\rangle^{MC(\calI)}$, 
we need to produce a point in $\langle x_i,x_i\rangle^{MC(\calI)}$ 
that is not satisfied by $\langle x_i\rangle^{MC_q(\calI)}$.} 

\kaarthik{
Consider the family of points $$ x_i = \frac 12 \left(\bm x_i^l \cdot \bm e_j + \bm x_i^u \cdot \bm e_j\right), ~~\widehat{x}_i = \left( \bm x_i^l \cdot \bm e_j \right)  \left(\bm x_i^u \cdot \bm e_j\right) \quad \forall j \in 1,\dots, |\calI|$$ 
where, $\bm e_j$ is a unit vector whose $j$\textsuperscript{th} component takes a value $1$. This family of points is satisfied by 
$\langle x_i,x_i\rangle^{MC(\calI)}$ 
and are not contained in 
$\langle x_i\rangle^{MC_q(\calI)}$, completing the proof.}
\end{proof}

Given these definitions, we use $\calP^{\calI}$ to denote the piecewise relaxation of $\calP$ for a given $\calI$, where all the nonlinear monomial terms are replaced with their respective piecewise convex relaxations. More formally,
\begin{equation}
\begin{aligned}
\calP^{\calI}: \ \ \ & \underset{\bm{x},\bm{y}}{\text{minimize}} & &  f^{\calI}(\bm{x},\bm{y}) \\
&\text{subject to} & & \bm{g}^{\calI}(\bm{x},\bm{y}) \leqslant 0, \\
& & & \bm{h}^{\calI}(\bm{x},\bm{y}) \leqslant 0, \\
& & &  \bm{x}^l_i \cdot \bm{\widehat{y}}_i \leqslant x_i \leqslant  \bm{x}^u_i \cdot \bm{\widehat{y}}_i, \; \; \; \forall \; i = 1 \ldots n\\
& & & \bm{y,\widehat{y}} \in \{0,1\}
\end{aligned}
\end{equation}
where, $f^{\calI}, \bm{g}^{\calI}$ and $\bm{h}^{\calI}$ inherit the above defined piecewise relaxations should the functions be nonlinear. 
Also, we let $f^{\calI}(\sigma)$ denote the objective value of a feasible solution, $\sigma$, to $\calP^{\calI}$.

\section{Adaptive Multivariate Partitioning Algorithm}
\label{sec:algo}

This section details the Adaptive Multivariate Partitioning (AMP) algorithm to compute global\footnote{global optimum is defined numerically by a tolerance, $\epsilon$.} optimal solutions to MINLPs. 

The effectiveness of AMP stems from the observation that the local optimal solutions found by the local solvers are often global optimum or are very close to the global optimum solution on standard benchmark instances. This observation was also made in the literature for the optimal power flow problem in power grids \cite{hijazi2016,kocuk2016strong,LuPSCC2017}.
AMP exploits this structure and adds sparse, spatial partitions to the variable domains around the local optimal solution. \harsha{It is important to note that though the partitions are dynamically added around the local optimal point (in the initial iterations), the AMP algorithm \textit{does not} discount the fact that the global optimal solution can potentially lie in sparser regions and will eventually partition the domains which contain the global optimum.}

\kaarthik{A flow-chart informally describing the steps of AMP and a formal pseudo-code for AMP are given in Figure \ref{fig:flowchart} and Algorithm \ref{alg:amp}, respectively.}
\begin{figure}
   \centering
   \includegraphics[scale=0.682]{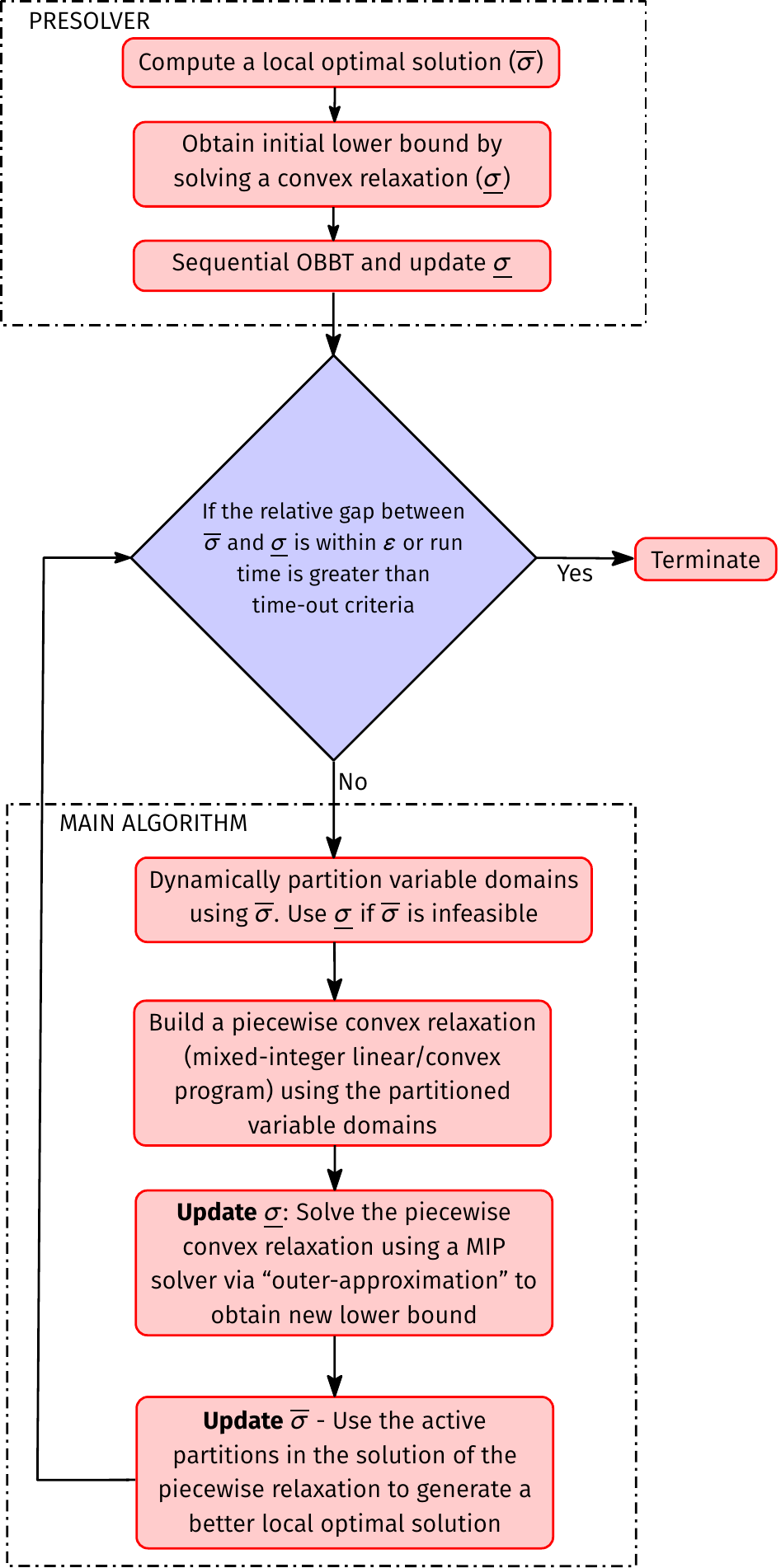}
   \caption{\harsha{Flow-chart describing the overall structure of AMP. The flow chart assumes that the MINLP is a feasible minimization problem.}}
   \label{fig:flowchart}
\end{figure}
\kaarthik{AMP consists of two main components. The first component is a presolve (see lines \ref{amp:ub} -- \ref{amp:lb}). The presolve component of the algorithm is sub-divided into four parts: (i) computing an initial feasible solution, $\overline{\sigma}$, (line \ref{amp:ub}), (ii) creating an initial set of partitions, $\calI$, 
(iii) sequential OBBT (line \ref{amp:tightenbounds}), and (iv) computing an initial lower bound, $\underline{\sigma}$, using the relaxations detailed in Sec. \ref{sec:problem} (line \ref{amp:lb}). }

\kaarthik{The second component of AMP is the main loop (lines \ref{amp:while1}--\ref{amp:while2}) that updates the upper bound, $\overline{\sigma}$, and the lower bound, $\underline{\sigma}$, of $\calP$, until either of the following conditions are satisfied: the bounds are within $\epsilon$ or the computation time exceeds the limit. At each iteration of the main loop, the partitions are refined and the corresponding piecewise convex relaxation is solved to obtain a lower bound (lines \ref{amp:partition} and \ref{amp:lb2}). Similarly the upper bound is obtained using a local solver and updated if it improves the best upper bound computed thus far (lines \ref{amp:ub2} and \ref{amp:ub3}). In the following sections, we discuss each step of the algorithm in detail.
 }

\begin{algorithm}
	\caption{Global optimization using AMP algorithm}
	\label{alg:amp}
	\begin{algorithmic}[1]
	   \Input{$\calP$}
	    \State $\overline{\sigma} \gets$ \Call{Solve}{$\calP$} \label{amp:ub} \Comment{Compute local optimal solution}
	    \State $\calI \gets $ \Call{InitializePartitions}{$\calP,\overline{\sigma}$} \label{amp:initialpartition} \Comment{Initialize variable partitions}
	    \State $\bm{x}^l,\bm{x}^u \gets$ \Call{TightenBounds}{$\calP^{\calI},\overline{\sigma}$} \label{amp:tightenbounds} \Comment{Sequential OBBT}
	    \State $\underline{\sigma} \gets$ \Call{Solve}{$\calP^{\calI}$} \label{amp:lb} \Comment{Initial lower bound computation}
	    \While{$\left(\frac{f(\overline{\sigma}) - f^{\calI}(\underline{\sigma})}{f^{\calI}(\underline{\sigma})} > \epsilon\right) \ \mathrm{and} \  \left(\text{Time} < \text{TimeOut}\right)$}  \label{amp:while1}
	        \State $\calI \gets $ \Call{RefinePartitions}{$\calP^{\calI},\underline{\sigma}$} \label{amp:partition} \Comment{Adaptive partition refinement}
	        \State $\underline{\sigma} \gets$ \Call{Solve}{$\calP^{\calI}$} \label{amp:lb2} \Comment{Compute new lower bound}
	        \State $\widehat{\sigma} \gets$ \Call{Solve}{$\calP, \underline{\sigma}$} \label{amp:ub2} \Comment{Compute new local optimum}
	        \State $\overline{\sigma} \gets \arg\min_{\sigma \in \overline{\sigma} \cup \widehat{\sigma}} f(\sigma)$ \label{amp:ub3} \Comment{Update upper bound}
	    \EndWhile \label{amp:while2}
	    \Output{$\underline{\sigma},\overline{\sigma}$} \label{amp:return}
	\end{algorithmic}
\end{algorithm}


\kaarthik{\subsection{Presolve} \label{subsec:presolve}
The first step of the presolver of AMP is to compute a local optimal solution, $\overline{\sigma}$, to the MINLP (line \ref{amp:ub} of Algorithm \ref{alg:amp}) . This is done using off-the-shelf, open-source solvers that use primal-dual interior point methods in conjunction with a branch-and-bound search tree to handle integer variables. This local solution, $\overline{\sigma}$, is further used to initialize the partitions, $\calI$ (line \ref{amp:initialpartition}). When the local solver reports infeasibility, we set the initial value of $f(\overline{\sigma}) = \infty$ (line \ref{amp:while1}) and use the solution obtained by solving the unpartitioned convex relaxation of $\calP$ to initialize $\calI$. In the subsequent sections we detail the partition initialization schemes and the sequential OBBT algorithm in lines \ref{amp:initialpartition} and \ref{amp:tightenbounds} of Algorithm \ref{alg:amp}.}

\subsubsection{Partition Initialization Scheme and Sequential OBBT} 
\label{subsubsec:partition_initialization}
\kaarthik{This section details the algorithm used in lines \ref{amp:initialpartition} and \ref{amp:tightenbounds} of AMP's presolve \emph{i.e.} \Call{InitializePartitions}{$\calP,\overline{\sigma}$} and \Call{TightenBounds}{$\calP^{\calI},\overline{\sigma}$}, respectively. The sequential OBBT procedure implemented in these functions is one of the key features of AMP. In many engineering applications there is little or no information about the lower and upper bounds ($\bm{x}^L,\bm{x}^U$) of the decision variables in the problem. Even when known, the gap between the bounds is often large and weaken relaxations. In practice, replacing the original bounds with tighter bounds can (sometimes) dramatically improve the quality of these relaxations. The basic idea of OBBT is the derivation of (new) valid bounds to improve the relaxations.
Though, OBBT is a well-known procedure used in global optimization, the key difference is that we apply OBBT sequentially by using \emph{mixed-integer models to tighten the bounds}.}

 \begin{algorithm}[h!]
 \caption{Partition Initialization Scheme}
 \label{alg:partition}
 \begin{algorithmic}[1]
     \Function{InitializePartitions}{$\calP, \overline{\sigma}$}
         \For{$i \in 1 \ldots n$}
                 \State $\calI_i \gets \{\langle x_i^L, x_i^U \rangle\}$ \label{init_partitions}        
             \EndFor
         \If {Bound-tightening without partitions}
             \State \Return $\calI$\label{part_bt}
         \Else \Comment{Bound-tightening with partitions}
                 \State \Return \Call{RefinePartitions}{$\calP^{\calI}, \overline{\sigma}$} \label{part_pbt}
         \EndIf
     \EndFunction
 \end{algorithmic}	
 \end{algorithm}


\kaarthik{For the sequential OBBT algorithm we present two procedures. First, OBBT without partitions (BT), which is
equivalent to partitioning with a single partition.
Second, partition-based OBBT (PBT) which uses the standard
partitioning approach described in Section \ref{subsubsec:adaptive}. For simplicity, we drop the term `optimization-based (OB)' in the acronym. The BT uses convex optimization problems to tighten the bounds, while the PBT uses convex MIPs to tighten the bounds. The two procedures differ in the initial set of partitions, $\calI$, used in the bound-tightening process. Hence, we first present two partition initialization schemes, one for the BT and another for the PBT, respectively. For bound-tightening without partitions, as the name suggests, the partition initialization scheme does not partition the variable domains \emph{i.e.}, $\calI_i \gets \{\langle x_i^L, x_i^U \rangle\}$ for every $i=1,\dots,n$ (line \ref{part_bt} of Algorithm \ref{alg:partition}). 
In the case of PBT, the partition initialization scheme initializes three partitions around the local optimal solution, $\overline{\sigma}$ with a user parameter $\Delta > 1$ (line \ref{part_pbt} of Algorithm \ref{alg:partition} and Section \ref{subsubsec:adaptive}). 
An illustration of the partitions added to a variable $x \in [x^L, x^U]$, whose value at the local optimal solution is $\overline{\sigma}(x)$, is shown in the Figure \ref{fig:pbt}.
If no local solution is obtained in line \ref{amp:ub} of Algorithm \ref{alg:amp}, the solution obtained by solving the unpartitioned convex relaxation of $\calP$ is used in place of $\overline{\sigma}(x)$.}\footnote{To keep the algorithm notation simple, this detail is omitted from Algorithm \ref{alg:partition}.}
\begin{figure}
    \centering
    \includegraphics{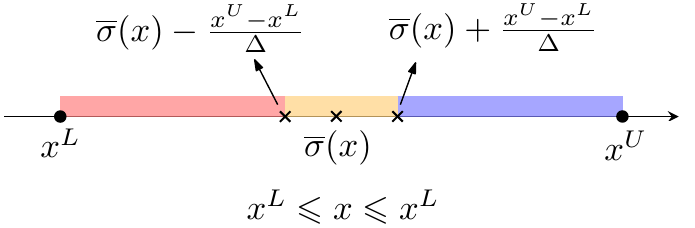}
    \caption{Partition initialization scheme for PBT.}
    \label{fig:pbt}
\end{figure}

\kaarthik{Once the initial set of partitions is computed, the sequential OBBT procedure iteratively computes new bounds by solving a modified version of $\calP^{\calI}$ (Algorithm \ref{Algo:bound}). Each iteration of the sequential OBBT algorithm proceeds as follows: for each continuous variable $x_{i}$ in $\calP$ appearing in the nonconvex terms, two problems $\calP^{\calI}_l$ and $\calP^{\calI}_u$ are solved, where $x_{i}$ is minimized and maximized, respectively, i.e.,}

\begin{subequations}
\begin{eqnarray}
\calP^{\calI}_l, \calP^{\calI}_u: \ \ \  \underset{\bm{x},\bm{y},\widehat{\bm{y}}}{\text{minimize}} & & \Mypm \,x_i \\
\text{subject to} & & 
f^{\calI}(\bm{x},\bm{y}) \le f(\overline{\sigma}) \label{tight1} \\
& & \bm{g}^{\calI}(\bm{x},\bm{y}) \leq 0, \\
& &  \bm{h}^{\calI}(\bm{x},\bm{y}) = 0, \\
& &  \bm{x}^l_i \cdot \widehat{\bm{y}}_i \leq x_i \leq  \bm{x}^u_i \cdot \widehat{\bm{y}}_i, \;\;\; \forall \; i=1\ldots n\\
& &  \bm{y},\widehat{\bm{y}} \in \{0,1\}  \label{t:1}
\end{eqnarray}
\label{eq:sbt}
\end{subequations}


\noindent
\kaarthik{where $\Mypm \,x_i$ in formulation \eqref{eq:sbt} denotes two optimization problems, where $x_i$ and $-x_i$ are individually minimized (lines \ref{bound:lb}-\ref{bound:ub} of Algorithm \ref{Algo:bound}). 
In both cases, a constraint that bounds the original objective function of $\calP^{\calI}$ with a best known feasible solution $\overline{\sigma}$ (equation \eqref{tight1}) is added when an initial feasible solution is available. Inclusion of this constraint on the objective is referred to as optimality/optimization-based bound-tightening in the literature. 
The OBBT algorithm stops when cumulative bound improvement between successive iterations, measured in terms of the $\ell^{\infty}$ norm, falls below specified tolerance values (line \ref{bound:while1} of Algorithm \ref{Algo:bound}).
We note that the sequential OBBT algorithm is naturally parallelizable as all the optimization problems are independently solvable.}

 \begin{algorithm}[h!]
 \caption{Sequential bound-tightening of $\bm{x}$}
 \label{Algo:bound}
 \begin{algorithmic}[1]
 	    \Function{TightenBounds}{$\calP^{\calI}, \overline{\sigma}$}
             \State $\widehat{\bm{x}}^l = \widehat{\bm{x}}^u \gets \bm{0}$  
             \While{$||\bm{x}^l - \widehat{\bm{x}}^l||_{\infty} > \epsilon^l$ and $||\bm{\bm{x}}^u - \widehat{\bm{x}}^u||_{\infty} > \epsilon^u$} \label{bound:while1}
                 \State $\widehat{\bm{x}}^l \gets  \bm{x}^l, \ \widehat{\bm{x}}^u \gets \bm{x}^u$
                 \For{$i=1,\ldots,n$}  \label{bound:for1}
                     \State $\sigma_i^l \gets $ \Call{Solve}{$\calP^{\calI}_l$} \label{bound:lb}
                     \State $\sigma_i^u \gets $ \Call{Solve}{$\calP^{\calI}_u$} \label{bound:ub}
                     \State $\bm{x_i}^l \gets \max(\sigma_i^l(x_i), \bm{x_i}^l)$, $\bm{x_i}^u \gets \min(\sigma_i^u(x_i), \bm{x_i}^u)$ \label{bound:new}
                 \EndFor \label{bound:for2}
             \EndWhile \label{bound:while2}
 \State \Return $\bm{\bm{x}}^l,\bm{\bm{x}}^u$.
 	    \EndFunction
 \end{algorithmic}	
 \end{algorithm}

\begin{remark}
OBBT procedures, usually referred to as domain reduction techniques, are well known in the global optimization literature \cite{puranik2017domain,horst2013handbook}. Our algorithm generalizes these techniques by performing optimality-based bound-tightening iteratively based on MIP-formulations (with piecewise convex relaxations) until convergence to a fixed point is achieved.
\end{remark} 

\subsection{Main Algorithm} \label{subsec:main}
\kaarthik{The main loop of AMP, shown in lines \ref{amp:while1} -- \ref{amp:while2} in Algorithm \ref{alg:amp} and main algorithm block of the flow chart in Figure \ref{fig:flowchart}, performs three main operations: 
\begin{enumerate}
\item First, the variable domains are refined adaptively, in a non-uniform manner (line \ref{amp:partition} of Algorithm \ref{alg:amp}).  
\item A piecewise convex relaxation is constructed using the partitioned domain, which is solved via outer-approximation to obtain an updated lower bound (line \ref{amp:lb2} of Algorithm \ref{alg:amp})
\item A local solve of the original, MINLP $\calP$, with the variables bounds restricted to the partitions obtained by the lower bounding solution, is performed to obtain a new local optimal solution. If this solution is better than the current incumbent, then the incumbent is updated (lines \ref{amp:ub2} and \ref{amp:ub3} of Algorithm \ref{alg:amp}). 
\end{enumerate}
In the following subsections, we detail each of the three operations involved in the main loop of AMP. }

\kaarthik{\subsubsection{Variable Domain Partitioning} 
\label{subsubsec:adaptive}
One of the core contributions of the AMP algorithm is the adaptive and non-uniform variable partitioning scheme. This part of the algorithm 
determines how variables in nonconvex terms of the original MINLP are partitioned. Existing approaches partition each variable domain uniformly into a finite number of partitions and the number of partitions increases with the number of iterations \cite{grossmann2013systematic,bergamini2008improved,castro2015tightening,hasan2010piecewise}. While this is a straight-forward approach for partitioning the variable domains, it potentially creates huge number of partitions far away from the global optimal value of each variable, \emph{i.e.}, many of the partitions are not useful. Though there have been methods, including sophisticated bound propagation techniques and logarithmic encoding to alleviate this issue, they do not scale to large-scale MINLPs. Instead, our approach successively tightens the relaxations with sparse, non-uniform partitions. This approach focuses partitioning on regions of the variable domain that appear to influence optimality the most. These regions are defined by a solution, $\sigma$, that is typically the lower bound solution at the current iteration.}

\kaarthik{The adaptive domain partitioning algorithm of AMP refines the variable partitions at a given iteration with a user parameter $\Delta > 1$. $\Delta$ is used to control the size and number of the partitions and influences the rate of convergence of the overall algorithm. The algorithm is similar to the algorithm used for initializing the partitions in Section \ref{subsubsec:partition_initialization} (see Figure \ref{fig:pbt}). It differs by using the lower bound solution obtained by solving the piecewise relaxations to refine the partition. Again, for the sake of clarity, Figure \ref{fig:partitions} illustrates the bivariate partition refinement for a bilinear term produced by the adaptive domain partitioning algorithm at successive iterations of the main algorithm. Figure \ref{fig:mcc_region}(a) geometrically illustrates the tightening of piecewise convex envelopes induced by the adaptive partitioning scheme. The pseudo-code of the domain partitioning algorithm is given in Algorithm \ref{Algo:partition}. The algorithm takes the current variable partitions and a lower bound solution as input and outputs a refined set of partitions for each variable. It first identifies, in line \ref{partition:variable} of Algorithm \ref{Algo:partition}, the partition where the lower bound solution is located (active partition) and splits that partition into three new partitions, whose sizes are defined by $\Delta$ and the size of the active partition (lines \ref{gamma1}-\ref{partition:split} in Algorithm \ref{Algo:partition}).}

\harsha{\paragraph{Exhaustiveness of variable domain partitioning scheme} 
Exhaustiveness of a partitioning scheme is one of the important requirements for any global optimization algorithm \cite{horst2013global}. The exhaustiveness of the adaptive domain partitioning scheme is built into AMP through the lines \ref{partition:conv1} -- \ref{partition:conv2} in Algorithm \ref{Algo:partition}. These conditions ensure that the AMP algorithm does not get stuck at a particular region of the variable. Instead, it guarantees that the largest partition outside the active partition is further refined when the active partition's width is less than a partition-width tolerance value, $\epsilon^p$. Here, the largest inactive partition of a given variable is analogous to the largest unexplored domain for that variable. Thus, the partitioning scheme in the AMP algorithm satisfies the desirable exhaustiveness property.}

\begin{algorithm}[htp]
\caption{Variable Domain Partitioning}
\label{Algo:partition}
\begin{algorithmic}[1]
    \Function{RefinePartitions}{$\calP^{\calI}, \sigma$}
        \For{$i \in 1 \ldots n$} 
            \State $k \gets \arg\max \sigma(\widehat{y}^k_i)$ \label{partition:variable} \Comment{Identifying the active partition for variable $i$}
            \State $\langle l_i, u_i \rangle \gets \calI_i^k$ 
            \State $\xi_i \gets \frac{u_i-l_i}{\Delta}$ \label{partition:size}
            \If{$\xi_i > \epsilon^p$} \Comment{Active partition is split into three partitions}
                \State $\gamma_1 \gets \min(l_i,\max(\sigma(x_i) - \xi_i, x_i^L))$\label{gamma1}
                \State $\gamma_2 \gets \max(l_i,\sigma(x_i) - \xi_i)$ \label{gamma2}
                \State $\gamma_3 \gets \min(u_i,\sigma(x_i) + \xi_i)$ \label{gamma3}
                \State $\gamma_4 \gets \max(u_i,\min(\sigma(x_i) + \xi_i, x_i^U))$\label{gamma4}
                \State $\calI_i \gets (\calI_i \setminus \langle l_i, u_i \rangle) \cup \langle \gamma_1, \gamma_2 \rangle \cup \langle \gamma_2, \gamma_3\rangle \cup \langle \gamma_3, \gamma_4 \rangle$ \label{partition:split}
            \Else \Comment{Partitioning of largest inactive partition}
                \State $\langle l_i, u_i  \rangle \gets \arg\max_{\calI_i} u_i - l_i$ \label{partition:conv1}
                \State $\calI_i \gets (\calI_i \setminus \langle l_i, u_i \rangle) \cup
                \langle l_i, l_i + \frac{u_i-l_i}{2} \rangle \cup \langle l_i + \frac{u_i-l_i}{2}, u_i \rangle$ \label{partition:conv2}                  
            \EndIf
        \EndFor
        \State \Return $\calI$
    \EndFunction
\end{algorithmic}	
\end{algorithm}

\begin{figure}[htp]
   \centering
   \subfigure[Iteration $k$]{
   \includegraphics[scale=0.82]{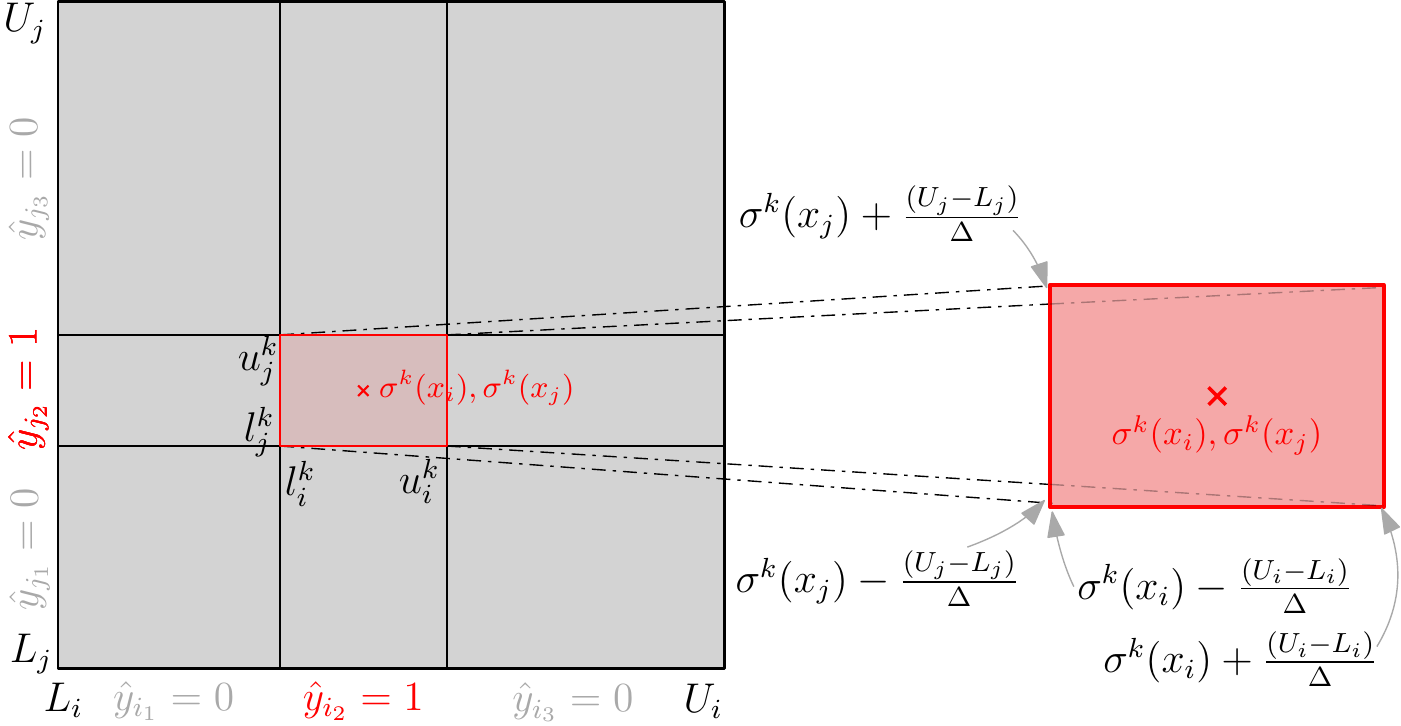}}
   \subfigure[Iteration $k+1$]{
   \includegraphics[scale=0.82]{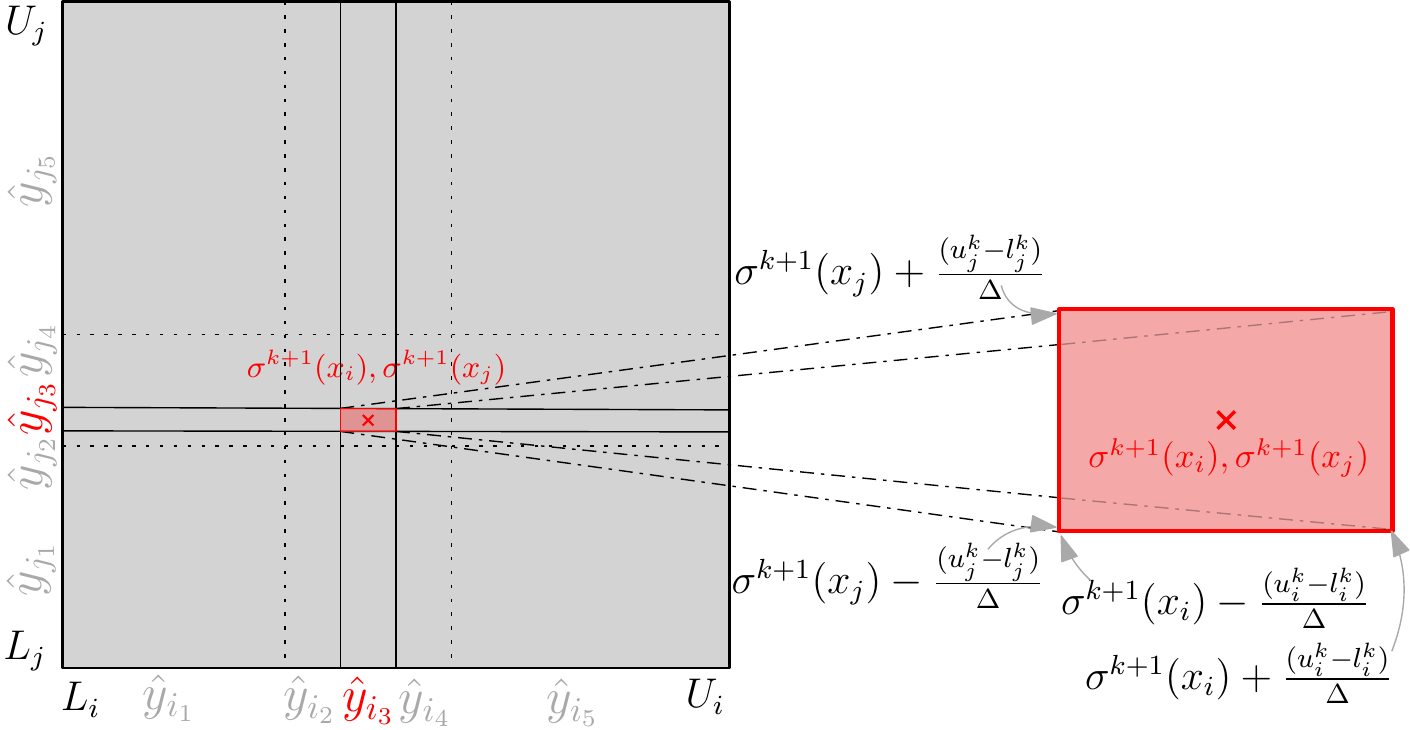}} 
   \caption{Adaptive partitioning strategy for a bilinear function $x_ix_j$ as described in Algorithm \ref{Algo:partition}. Red and gray colored boxes represent active and inactive partitions, respectively. $k$ and $k+1$ refer to successive solutions used to partition the active partition.}
   \label{fig:partitions}
\end{figure}

\kaarthik{\subsubsection{Computing Lower and Upper Bounds} \label{subsubsec:lb_ub}
Once the variable partitions are refined, the main loop of the AMP algorithm constructs and solves a piecewise convex relaxation, $\calP^{\calI}$ (line \ref{amp:lb2} of Algorithm \ref{alg:amp}). Each $\calP^{\calI}$ is a convex MIP where all the constraints are either linear or second-order cones (SOCs). Theoretically, $\calP^{\calI}$ can be solved by off-the-shelf mixed-integer, conic solvers. However, initial computational experiments suggested that several moderately sized problems with SOC constraints were difficult to solve, even with state-of-the-art solvers. It was the case that either the solver convergence was very slow or the solve terminated with a numerical error. To circumvent this issue and solve these convex MIPs in a computationally efficient manner, the SOC constraints in the convex MIP were outer-approximated via first-order approximations, using the lazy-callback feature in the MIP solvers.}

To obtain a new local optimal solution in the main loop of AMP (line \ref{amp:ub2} of Algorithm \ref{alg:amp}), the feasible solution is obtained by solving the NLP, $\calP^u$, shown in Eq. \eqref{eq:Pu}. $\calP^u$ is constructed at each iteration of the main loop using the original MINLP, $\calP$, and the lower bound solution computed at that iteration, $\underline{\sigma}$. 
\begin{subequations}
\label{eq:Pu}
\begin{eqnarray}
\calP^u: \ \ \  \underset{\bm{x},\bm{y}}{\text{minimize}} & &  f(\bm{x},\bm{y}) \\
\text{subject to} & & \bm{g}(\bm{x},\bm{y}) \leq 0, \\
& &  \bm{h}(\bm{x},\bm{y}) = 0, \\
& &  \bm{x}^l_i \cdot \underline{\sigma}(\widehat{\bm{y}}_i) \leq {x}_i \leq \bm{x}^u_i \cdot \underline{\sigma}(\widehat{\bm{y}}_i), \; \; \; \forall \; i=1\ldots n \label{ub:bounds}\\
& &  \bm{y} = \underline{\sigma}(\bm{y}) \label{ub:fix} 
\end{eqnarray}
\end{subequations}

\noindent where constraint \eqref{ub:bounds} forces the variable assignments into the partition defined by the current lower bound. Constraint \eqref{ub:fix} fixes all the original binary variables to the lower bound solution. $\calP^u$ is then solved to local optimality using a local solver. The motivation for this approach is based on empirical observations that the relaxed solution is often very near to the global optimum solution and that the NLP is solved fast once all binary variables are fixed to constant values. $\calP^u$ is essentially a projection of the relaxed solution (lower bound solution, $\underline\sigma$) back onto a near point in the feasible region of $\calP$; this approach is often used for recovering feasible solutions \cite{bergamini2008improved}. In the forthcoming Lemma \ref{lem:1}, we claim that the value of the objective function of the piecewise convex relaxation at each iteration monotonically increases to the global optimum solution with successive partition refinements. 

\begin{lemma}
\label{lem:1}
Let $\underline{\sigma}^{\calI}$ denote the optimal solution to the formulation $\calP^{\calI}$ and let $\sigma^*$ denote the global optimal solution to $\calP$. Then, $f^{\calI}(\underline{\sigma}^{\calI})$ monotonically increases to $f(\sigma^*)$ as $|\calI_i| \rightarrow \infty$ for every $i = 1,\dots,n$. 
\end{lemma}
\begin{proof}
Without loss of generality, we assume $\calP$ is feasible and restrict our discussion to bilinear terms. Let $x_ix_j$ be a bilinear term. Given a finite set of partitions, that is, $1 \leqslant |\calI_i|,|\calI_j| < \infty$, there always exists a partition in $\calI_i$ and $\calI_j$ that is active in the solution to $\calP^{\calI}$. Let the active partitions have lengths $\epsilon^l_i+\epsilon^u_i$ and $\epsilon^l_j+\epsilon^u_j$, respectively. Given the exhaustiveness property of the adaptive partitioning scheme discussed in section \ref{subsubsec:adaptive}, assume the active partition contains the global optimum solution $\sigma^*({\bf x}^*, {\bf y}^*)$\footnote{Exhaustiveness of the partitioning scheme implies AMP will eventually partition all other domains small enough such that AMP will pick an active partition with the global optimal whose length is $\le \epsilon^l_i+\epsilon^u_i$. }. 
Then, we have
$$x^*_i - \epsilon^l_i \leqslant x_i \leqslant x^*_i + \epsilon^u_i, \quad x^*_j - \epsilon^l_j \leqslant x_j \leqslant x^*_j + \epsilon^u_j.$$
For these active partitions, the McCormick constraints in \eqref{eq:mcc1} and \eqref{eq:mcc3} linearize $x_ix_j$ as follows:
\begin{align}
    \widehat{x_{ij}} &\geqslant (x^*_i - \epsilon^l_i)x_j + (x^*_j - \epsilon^l_j)x_i - (x^*_i - \epsilon^l_i) (x^*_j - \epsilon^l_j) \label{eq:lem1a}\\
    &= (x_i^*x_j + x_j^*x_i - x_i^*x_j^*) + \underbrace{\epsilon^l_i(x_j^*-x_j) + \epsilon^l_j(x_i^*-x_i) - \epsilon^l_i\epsilon^l_j}_{\mathcal{E}(\epsilon_i^l,\epsilon_j^l)} \nonumber
\end{align}
and 
\begin{align}
    \widehat{x_{ij}} &\leqslant (x^*_i - \epsilon^l_i)x_j + (x^*_j + \epsilon^u_j)x_i - (x^*_i - \epsilon^l_i) (x^*_j + \epsilon^u_j) \label{eq:lem1b}\\
    &= (x_i^*x_j + x_j^*x_i - x_i^*x_j^*) + \underbrace{\epsilon^l_i(x_j^*-x_j) + \epsilon^u_j(x_i-x_i^*) + \epsilon^l_i\epsilon^u_j}_{\mathcal{E}(\epsilon_i^l,\epsilon_j^u)} \nonumber
\end{align}
where, $\mathcal{E}(\epsilon_i^l,\epsilon_j^l) < 0$ and $\mathcal{E}(\epsilon_i^l,\epsilon_j^u) > 0$ are the error terms of the under- and over-estimator, respectively. It is trivial to observe that the error terms themselves constitute the McCormick envelopes of $(x_i^*-x_i)(x_j^*-x_j)$ if one were to linearize this product. Further, observing that the error terms, as described above, are parametrized by the size of the active partition containing the global optimum point, given by $\epsilon_i^l+\epsilon_i^u$ and $\epsilon_j^l+\epsilon_j^u$, for variables $x_i$ and $x_j$, respectively, we now derive the analytic forms of these terms as a function of the total number of partitions for the adaptive partitioning case. 

Line \ref{amp:partition} in Algorithm \ref{alg:amp} ensures that the partitions created during the $(k+1)^{th}$ iteration of the main loop for either of the variables, $x_i$ or $x_j$, is a subset of the partitions created for the corresponding variables in the $k^{th}$ iteration. Also, for a $k^{th}$ iteration of an adaptive refinement step for variables $x_i$ and $x_j$, we assume that at most $3+2(k-1)$  partitions exist within the given variable bounds ($L_i,U_i$) and ($L_j, U_j$), respectively. Thus, the length of the above mentioned active partition which contains the global solution is given by 
$$\epsilon_i^l + \epsilon_i^u = \frac{U_i-L_i}{\Delta^{\frac{|\calI_i|-1}{2}}}, \quad \epsilon_j^l + \epsilon_j^u = \frac{U_j-L_j}{\Delta^{\frac{|\calI_j|-1}{2}}}.$$
\noindent
Clearly as $|\calI_i|$ and $|\calI_j|$ approach $\infty$, the error terms of the McCormick envelopes, $\mathcal{E}(\epsilon_i^l, \epsilon_j^l)$ and $\mathcal{E}(\epsilon_i^l,\epsilon_j^u)$, approach zero, and thus enforcing $x_i=x_i^*$, $x_j=x_j^*$ and $\widehat{x_{ij}} = x_i^*x_j^*$. Therefore, as $|\calI_i|$ approaches $\infty$ for every variable $i$ that is partitioned, $f^{\calI}(\underline{\sigma}^{\calI})$ approaches the global optimal solution $f(\sigma^*)$.  

Also observe that since  $\mathfrak{F}(\calP^{\calI}) \subset \mathfrak{F}(\calP)$ for any finite $\calI$, $f^{\calI}(\underline{\sigma}^{\calI}) \leqslant f(\sigma^*)$. Here, $\mathfrak{F}(\cdot)$  denotes the feasible space of the formulation of ``$\cdot$". Furthermore, the partition set in iteration $k$ is a proper subset of the previous iteration's partitions. This proves the monotonicity of the sequence of values  $f^{\calI}(\underline{\sigma}^{\calI})$ with increasing iterations of the main loop of AMP. 
\end{proof}

\section{Computational results}
\label{sec:results}
In the remainder of this paper, we refer to AMP as Algorithm \ref{alg:amp} without the implementation of line 3, \emph{i.e.}, without any form of sequential OBBT. BT-AMP and PBT-AMP refer to Algorithm \ref{alg:amp} implemented with bound-tightening without and with partitions added, respectively. 
The performance of these algorithms is evaluated 
on a set of standard benchmarks from the literature with mutlilinear terms. 
These problems include a small NLP that is used to highlight the differences between the sparse, adaptive approaches and uniform partitioning approaches. The details and sources for each problem instance are shown later in this section in Table \ref{tab:data}. In this table, we also mention the continuous variables in mutlilinear terms chosen for partitioning\footnote{See \cite{boukouvala2016global} for more details on strategies for choosing the variables for partitioning.}. 
Ipopt { 3.12.8} and Bonmin {1.8.2} are used as local NLP and MINLP solvers for the feasible solution computation in AMP, respectively.
MILPs and MIQCQCPs are solved using CPLEX 12.7 (cpx) and/or Gurobi 7.0.2 (grb) with default options and presolver switched on. 
The outer-approximation algorithm was implemented 
using the lazy callback feature of CPLEX and Gurobi. Given that the bound-tightening procedure consists of independently solvable problems, 10 parallel threads were used during bound-tightening. In the Appendix we provide a detailed sensitivity analysis of the parameters of AMP.

Every bound-tightening (BT and PBT) problem was solved to optimality (except \textit{meyer15}). For \textit{meyer15}, 0.1\% optimality gap was used as a termination criteria because it is a large-scale MINLP.
The value of $\epsilon$ and the ``TimeOut'' parameter in Algorithm \ref{alg:amp} were set to $0.0001$ and $3600$ seconds, respectively. However, for BT and PBT, we did not impose any time limit. Thus, for a fair comparison, we set the time limit for the global solver to be the sum of bound tightening time and 3600 seconds (denoted by $T^+$ in the tables).   
In the results, ``TO'' indicates that the AMP solve timed-out.
All results are benchmarked with BARON 17.1, a state-of-the-art global optimization 
solver \cite{sahinidis1996baron,tawarmalani2005polyhedral}. CPLEX 12.7 and Ipopt 3.12.8 are used as the underlying MILP and non-convex local solvers for BARON. JuMP, an algebraic modeling language in Julia \cite{dunning2017jump}, was used for implementing all the algorithms and invoking the optimization solvers. All the computational experiments were performed using the high-performance computing resources at the Los Alamos National Laboratory with Intel CPU E5-2660-v3, Haswell micro-architecture, 20 cores (2 threads per core) and 125GB of memory.

\subsection{Performance on a Small-scale NLP}
\begin{equation*}
\begin{aligned}
\mathit{NLP1:} \quad &  \underset{x_1,\ldots,x_8}{\text{minimize}} & &  x_1+x_2+x_3 \\
&\text{subject to} & & 0.0025(x_4+x_6)-1 \leq 0, \\
& & & 0.0025(-x_4+x_5+x_7) -1 \leq 0, \\
& & & 0.01(-x_5+x_8) -1 \leq 0, \\
& & & 100x_1 - x_1x_6 + 833.33252 x_4 - 83333.333 \leq 0, \\
& & & x_2x_4 - x_2x_7 - 1250x_4 + 1250x_5 \leq 0, \\
& & & x_3x_5 - x_3x_8  - 2500x_5 +1250000 \leq 0, \\
& & & 100 \leq x_1 \leq 10000, \\
& & & 1000 \leq x_2,x_3 \leq 10000, \\
& & & 10 \leq x_4,x_5,x_6,x_7,x_8 \leq 1000
\end{aligned}
\end{equation*}

In this section, we perform a detailed study of AMP with and without OBBT on \textit{NLP1}, a small-scale, continuous nonlinear program adapted from Problem 106 in \cite{hock1980test}. This small problem helps illustrate many of the salient features of AMP. \textit{NLP1} has gained considerable interest from the global optimization literature due its large variable bounds and weak McCormick relaxations. Since this is a challenging problem for uniform, piecewise McCormick relaxations, this problem has been studied in detail in \cite{castro2015tightening,castro2015normalized,teles2013univariate}. The value of the global optimum for \textit{NLP1} is 7049.2479 and the solution is  $x^*_i, \ i=1,\ldots,8 = [579.307, 1359.97, 5109.97, 182.018, 295.601, 217.982, 286.417, 395.601]$.

\subsubsection{AMP versus Uniform Partitioning on \textit{NLP1}}
Figure \ref{Fig:p3_adap_unif} compares AMP (without BT/PBT) with a uniform partitioning strategy that is often used in state-of-the-art solvers to obtain global solutions. For a fair comparison, we used the same number of partitions for both methods at every iteration. From Figure \ref{Fig:p3_adap_unif}(a), it is evident that AMP exhibits larger optimality gaps in the first few iterations (134\%, 44\%, 20\%, vs. 97\%, 44\%, 16\%, etc.). However, the convergence rate to global optimum is much faster with AMP (within 190.9 seconds). This behaviour is primarily attributed to the adaptive addition of partitions around the best-known local solution (also global in this case) instead of spreading them uniformly. 

\begin{figure}[htp]
  \centering
  \subfigure[Optimality gap]{
  \includegraphics[scale=0.71]{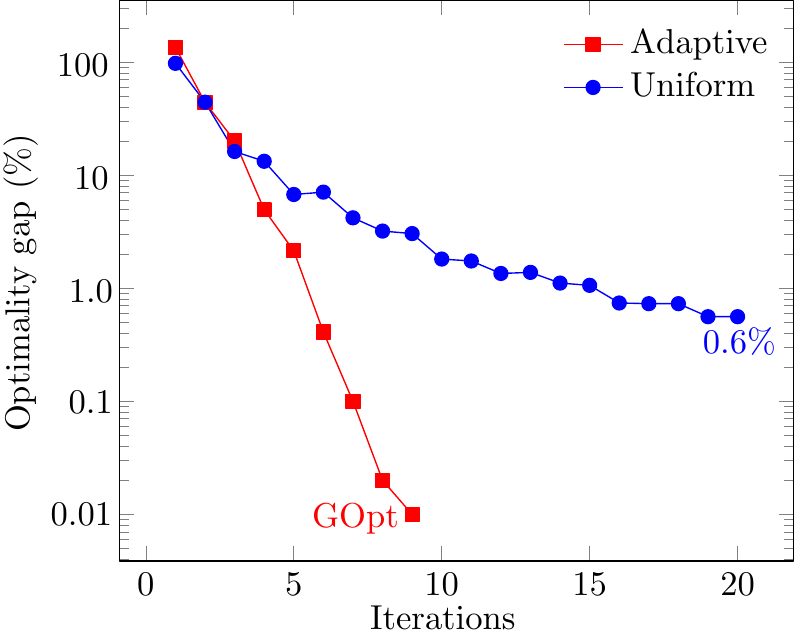}} 
  \subfigure[Cumulative time]{
  \includegraphics[scale=0.71]{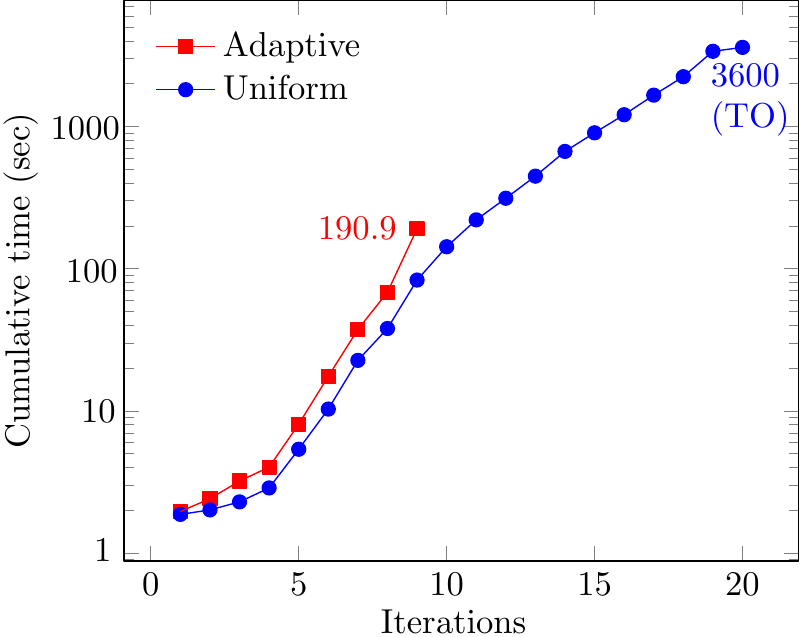}} 
  
  \caption{Performance of AMP ($\Delta=4$) and uniform partitioning on \textit{NLP1}. Note that the y-axis is on log scale.}
  \label{Fig:p3_adap_unif}
\end{figure}

\subsubsection{Performance of AMP on \textit{NLP1}}

Figure \ref{Fig:p3_partitions} shows the active partitions chosen by every iteration of AMP for $\Delta=4$. 
This figure illustrates the active partitions at each iteration of the main loop of AMP and the convergence of each variable partition to its corresponding global optimal value. One of the primary motivations of the adaptive partitioning strategy comes from the observation that every new partition added adaptively refines the regions (hopefully) closer to the global optimum values.
In Figure \ref{Fig:p3_partitions}(a), this behaviour is clearly evident on all the variables except $x_3$. 
Although the initial active partition on $x_3$ did not contain the global optimum, AMP converged to the global optimum value. The convergence time of AMP to the global optimum using Gurobi was \textit{190.9 seconds}. The total number of binary partitioning variables is 152 (19 per continuous variable).  

\begin{figure}[htp]
  \centering
  \subfigure[Variables $x_1,x_3$]{
  \includegraphics[scale=0.63]{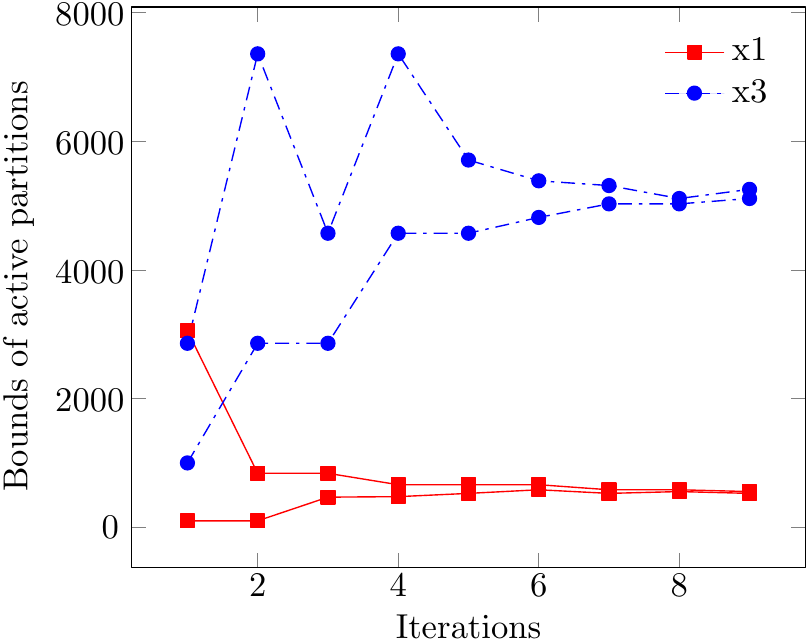}}       
  \subfigure[Variables $x_2,x_4$]{
  \includegraphics[scale=0.63]{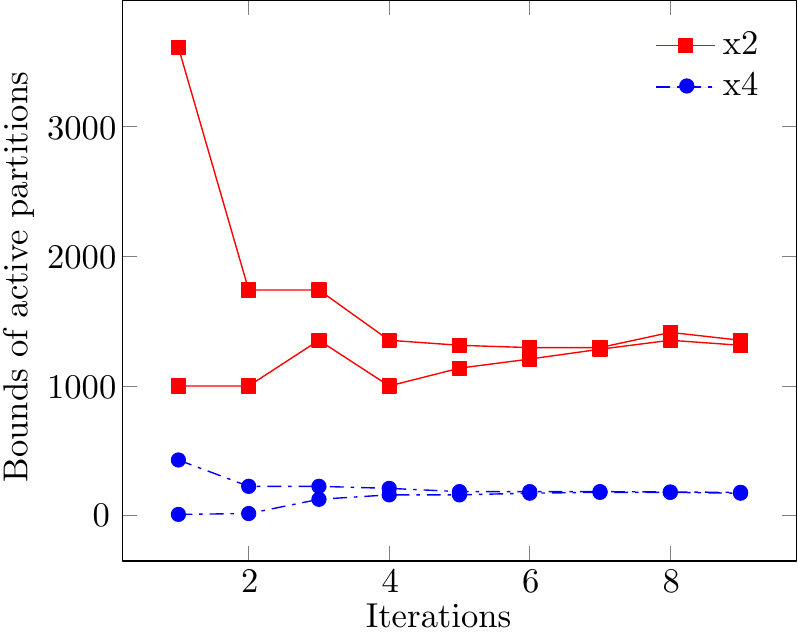}} 
  \subfigure[Variables $x_5,x_6$]{
  \includegraphics[scale=0.63]{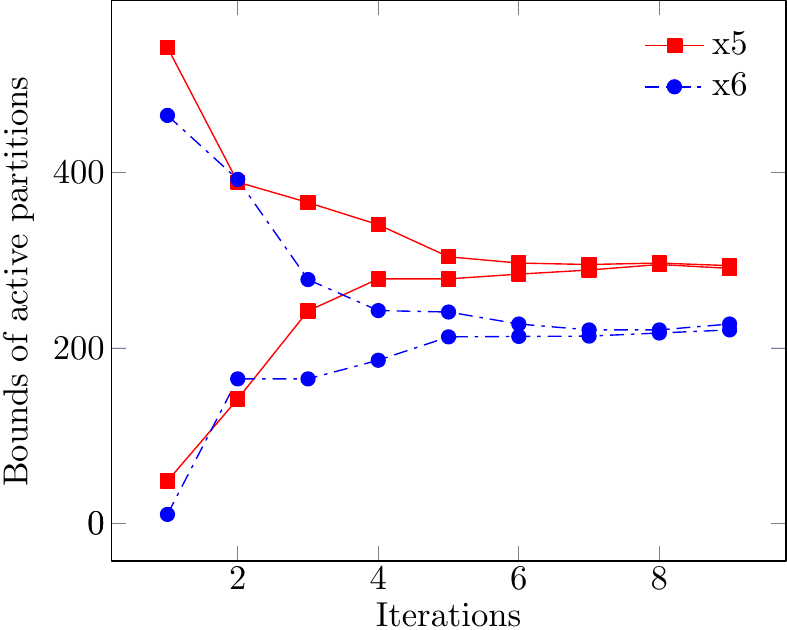}} 
  \subfigure[Variables $x_7,x_8$]{
  \includegraphics[scale=0.63]{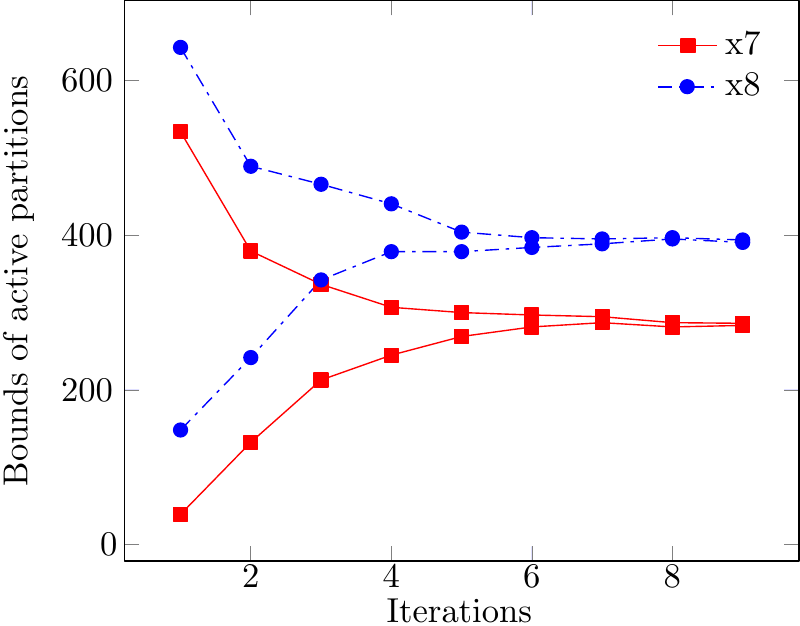}} 
  \caption{Upper and lower bounds of active partitions chosen by AMP algorithm (without OBBT) for the variables $x_i, i=1,\ldots,8$ of \textit{NLP1}.}
  \label{Fig:p3_partitions}
\end{figure}

\subsubsection{Benefits of MIP-based OBBT on \textit{NLP1}}
\begin{figure}[htp]
  \centering
  \subfigure[Tightened bounds after each iteration]{
  \includegraphics[scale=0.36]{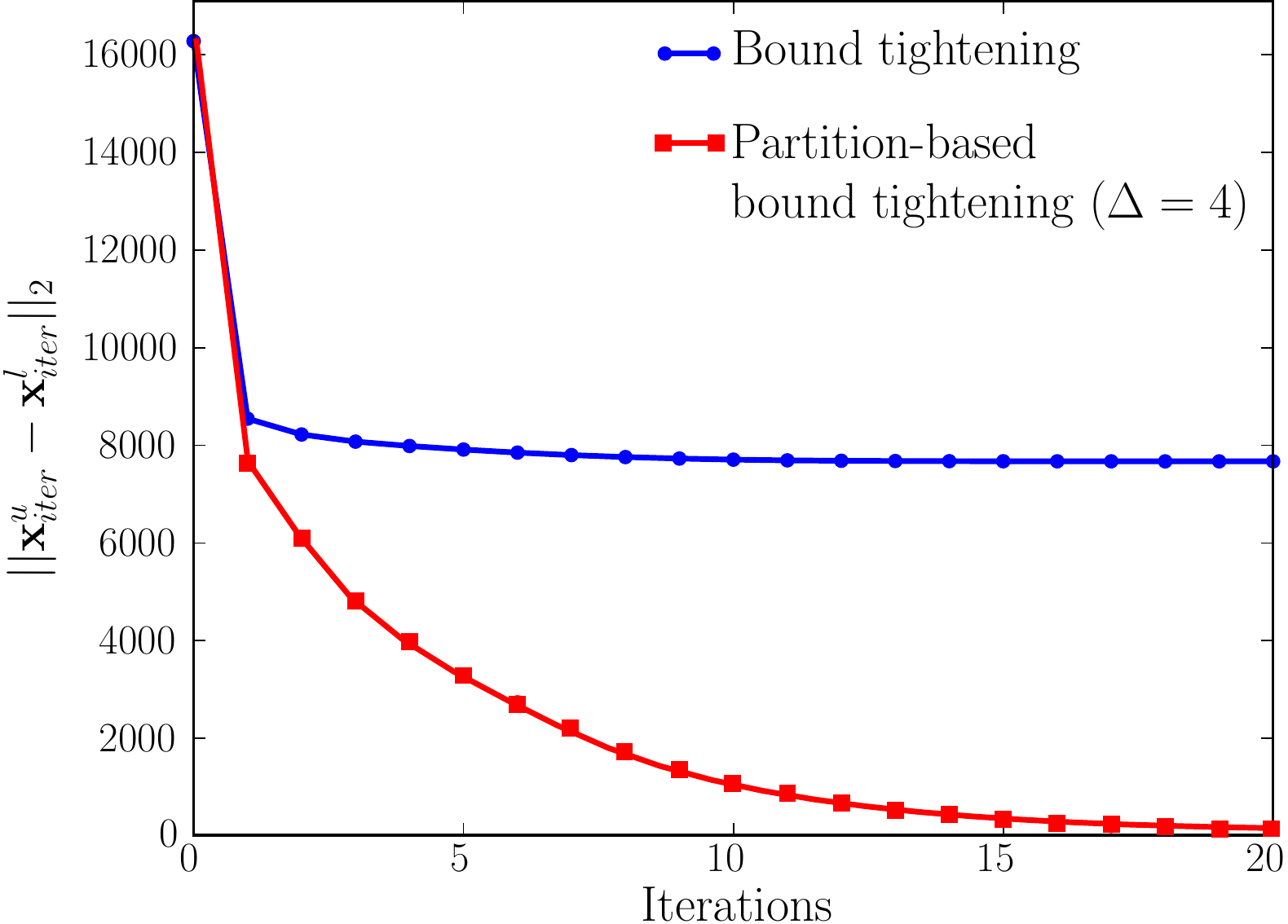}}       
  \subfigure[Time per iteration]{
  \includegraphics[scale=0.3685]{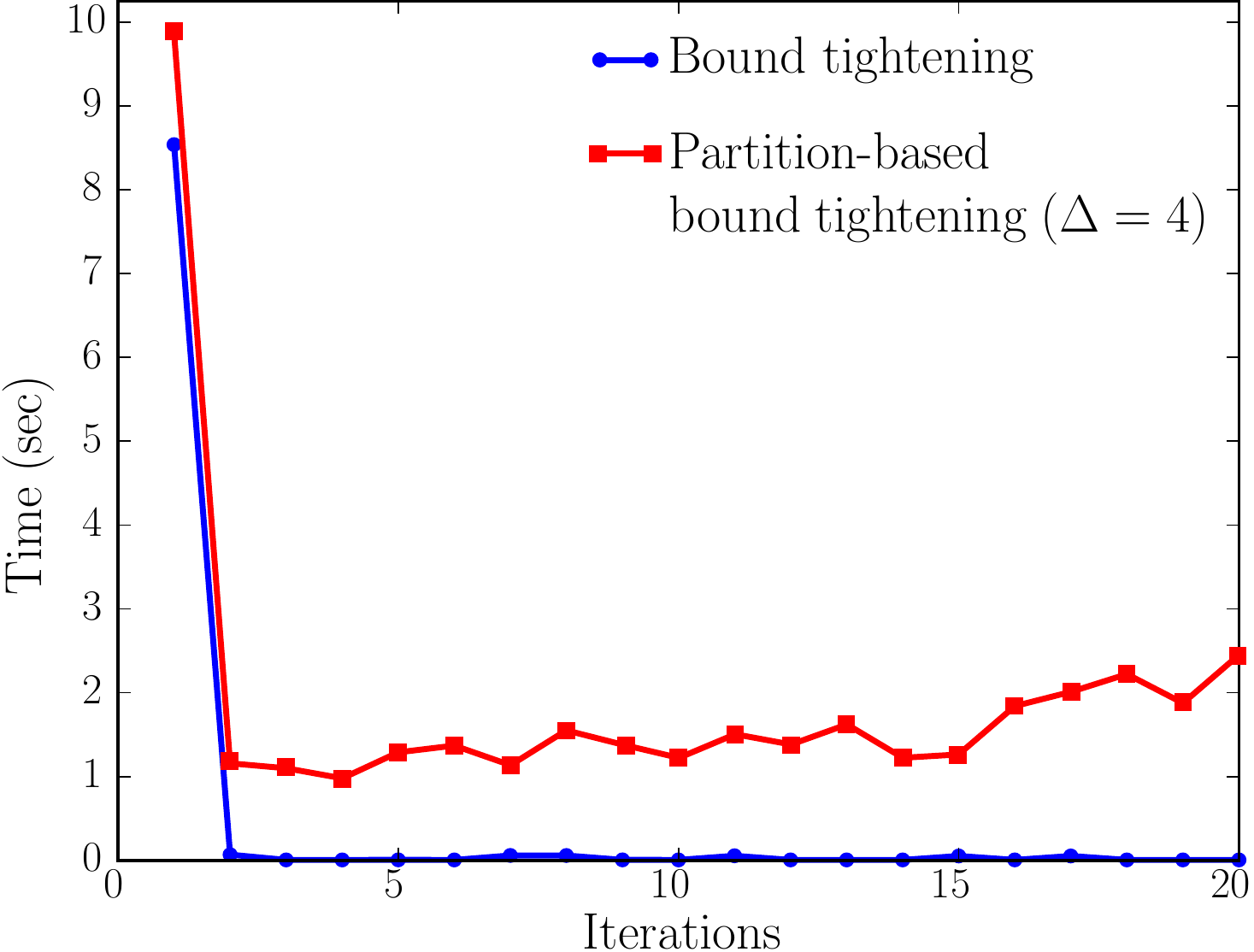}} 
  
  \caption{Performance of sequential BT and sequential PBT techniques on \textit{NLP1}.}
  \label{Fig:p3_BL}
\end{figure}

\begin{table}[htp]
\caption{Contracted bounds after applying sequential bound tightening to nlp1.}
\centering
\begin{tabular}{ccccccc}
\toprule
 & \multicolumn{2}{c}{Original bounds} & \multicolumn{2}{c}{PBT bounds} &
 \multicolumn{2}{c}{\#BVars added}\\
 \cmidrule(lr){2-3}
 \cmidrule(lr){4-5}
 \cmidrule(lr){6-7}
 Variable & $L$ & $U$ & $l$ & $u$ & BT, AMP & PBT, AMP \\
 &&&&&($\Delta=4$)&($\Delta=4$) \\
 \midrule
$x_1$ &100 &10000 &573.1 &585.1 &0, 14 &3, 3 \\
$x_2$ &1000 &10000 &1351.2 &1368.5 &0, 14 &3, 3\\
$x_3$ &1000 &10000 &5102.1 &5117.5 &0, 15 &3, 3\\
$x_4$ &10 &1000 &181.5 &182.5 &0, 15 &3, 3\\
$x_5$ &10 &1000 &295.3 &296.0 &0, 15 &3, 3\\
$x_6$ &10 &1000 &217.5 &218.5 &0, 15 &3, 3\\
$x_7$ &10 &1000 &286.0 &286.9 &0, 15 &3, 3\\
$x_8$ &10 &1000 &395.3 &396.0 &0, 15 &3, 3\\
\cmidrule(r){1-7}
Total &&&&&118&48 \\
\bottomrule
\end{tabular}
\label{tab:tighten}
\end{table}

Figure \ref{Fig:p3_BL} and Table \ref{tab:tighten} show the effectiveness of sequential BT and sequential PBT techniques on \textit{NLP1}. As expected, in Figure \ref{Fig:p3_BL}(a), the disjunctive polyhedral representation of the relaxed regions in PBT (around the initial local solution) drastically reduce the global bounds on the variables (to almost zero gaps). 
Figure \ref{Fig:p3_BL}(b) shows that PBT, even when solving a MILP in every iteration, does not incur too much computational overhead on a small-scale problem like \textit{NLP1}. 

A qualitative description of the improved performance is presented in Table \ref{tab:tighten}. {The column titled ``\#BVars added'' shows the total number of partitions that were added for each variable for BT-AMP and PBT-AMP. While BT adds no partitions and AMP adds a total of 118 partitions, PBT adds a total of 24 partitions in addition to 24 more partitions during AMP. As is seen in the results in the ``PBT bounds" column, the bounds of the variables are tightened to near global optimum values and AMP needs very few additional partitions to prove global optimality within 40 seconds. In contrast, AMP adds a lot more partitions as the bounds of the variables after BT are not tight enough.} Overall, for \textit{NLP1}, it is noteworthy that PBT-AMP outperforms most of the state-of-the-art piecewise relaxation methods developed in the literature.





\subsection{Performance of AMP on Large-scale MINLPs}
\label{subsec:perf_no_bnd}

\setlength\tabcolsep{4 pt}
\begin{table}[htp]
    \centering
    \scriptsize
    \caption{Summary of the performance of AMP without OBBT on all instances. 
    Here, we compare the run times of BARON, AMP with $\Delta=8$ (cpx), AMP with the best $\Delta$ (cpx) and AMP with the best $\Delta$ (grb). Values under ``Gap'' and ``$T$'' are in \% and seconds, respectively. ``Inf'' implies that the solver failed to provide a bound within the prescribed time limit. For each instance, the bold face font represents best run time or the best optimality gap (if the solve times out)}
    \begin{tabular}{l|rr|rr|rr|rrr|rrr}
\toprule
& \multicolumn{2}{c}{{COUENNE}} & \multicolumn{2}{|c}{BARON} & \multicolumn{2}{|c}{AMP-cpx} & \multicolumn{3}{|c|}{AMP-cpx} & \multicolumn{3}{c}{AMP-grb} \\
& \multicolumn{2}{c}{} & \multicolumn{2}{|c}{} & \multicolumn{2}{|c}{$\Delta=8$} & \multicolumn{3}{|c|}{$\Delta^*$} & \multicolumn{3}{c}{$\Delta^*$} \\
\cmidrule(lr){2-3} \cmidrule(lr){4-5} \cmidrule(lr){6-7} \cmidrule(lr){8-10} \cmidrule(lr){11-13}
Instances & Gap & $T$ & Gap & $T$ & Gap & $T$ & $\Delta$ & Gap & $T$ & $\Delta$ & Gap & $T$ \\
\midrule
p1 & {\bf GOpt} & {\bf 0.01} & GOpt & 0.02 & GOpt & 0.26 & 32 & GOpt & 0.19 & 32 & GOpt & 0.06 \\
p2 & {\bf GOpt} & {\bf 0.01} & {\bf GOpt} & {\bf 0.01} & GOpt & 0.05 & 16 & GOpt & 0.03 & 32 & GOpt & 0.10 \\
fuel & Inf & N/A & {\bf GOpt} & {\bf 0.03} & GOpt & 0.07 & 4 & {\bf GOpt} & {\bf 0.03} & 4 & GOpt & 0.05 \\
ex1223a & {\bf GOpt} & {\bf 0.01} & GOpt & 0.02 & {\bf GOpt} & {\bf 0.01} & 32 & {\bf GOpt} & {\bf 0.01} & 16 & GOpt & 0.02 \\
ex1264 & GOpt & 2.04 & GOpt & 1.44 & GOpt & 1.42 & 16 & GOpt & 0.90 & 8 & {\bf GOpt} & {\bf 0.79} \\
ex1265 & GOpt & 5.22 & GOpt & 13.30 & GOpt & 0.94 & 16 & {\bf GOpt} & {\bf 0.17} & 8 & GOpt & 0.28 \\
ex1266 & GOpt & 5.37 & GOpt & 10.81 & GOpt & 0.27 & 32 & {\bf GOpt} & {\bf 0.14} & 32 & GOpt & 0.16 \\
eniplac & GOpt & 128.82 & GOpt & 207.37 & GOpt & 1.17 & 32 & {\bf GOpt} & {\bf 0.68} & 32 & GOpt & 0.75 \\
util & GOpt & 14.56 & {\bf GOpt} & {\bf 0.10} & GOpt & 1.21 & 16 & GOpt & 0.54 & 16 & GOpt & 0.55 \\
meanvarx & GOpt & 1.06 & {\bf GOpt} & {\bf 0.05} & GOpt & 290.61 & 16 & GOpt & 95.51 & 16 & GOpt & 70.09 \\
blend029 & GOpt & 35.18 & GOpt & 2.46 & GOpt & 1.74 & 32 & {\bf GOpt} & {\bf 0.74} & 32 & GOpt & 1.00 \\
blend531 & 3.06 & TO & GOpt & 111.79 & GOpt & 185.33 & 8 & GOpt & 185.33 & 16 & {\bf GOpt} & {\bf 49.76} \\
blend146 & 4.19 & TO & 2.20 & TO & 2.01 & TO & 8 & 2.01 & TO & 8 & {\bf 1.60} & {\bf TO} \\
blend718 & 156.45 & TO & 175.10 & TO & {\bf GOpt} & {\bf 379.58} & 8 & {\bf GOpt} & {\bf 379.58} & 8 & GOpt & 581.68 \\
blend480 & 102.41 & TO & {\bf GOpt} & {\bf 326.95} & 0.32 & TO & 32 & 0.04 & TO & 16 & 0.02 & TO \\
blend721 & 0.60 & TO & GOpt & 548.90 & GOpt & 504.74 & 32 & GOpt & 256.77 & 16 & {\bf GOpt} & {\bf 176.11} \\
blend852 & 1.74 & TO & 0.08 & TO & GOpt & 750.88 & 4 & {\bf GOpt} & {\bf 169.24} & 16 & GOpt & 322.80 \\
wtsM2\_05 & GOpt & 3426.28 & {\bf GOpt} & {\bf 153.30} & 20.08 & TO & 8 & 20.08 & TO & 32 & GOpt & 386.95 \\
wtsM2\_06 & 31.95 & TO & {\bf GOpt} & {\bf 228.18} & 8.76 & TO & 4 & GOpt & 2395.71 & 32 & GOpt & 972.20 \\
wtsM2\_07 & {\bf GOpt} & {\bf 68.37} & GOpt & 759.96 & 0.10 & TO & 8 & 0.10 & TO & 16 & 0.54 & TO \\
wtsM2\_08 & 39.43 & TO & 388.62 & TO & 9.82 & TO & 4 & {\bf 5.45} & {\bf TO} & 4 & 7.92 & TO \\
wtsM2\_09 & 60.37 & TO & Inf & TO & 68.58 & TO & 10 & 36.47 & TO & 4 & {\bf 7.47} & {\bf TO} \\
wtsM2\_10 & 0.64 & TO & 76.48 & TO & 35.88 & TO & 32 & 24.95 & TO & 16 & {\bf 0.10} & {\bf TO} \\
wtsM2\_11 & 64.71 & TO & 107.56 & TO & 7.88 & TO & 16 & {\bf 3.50} & {\bf TO} & 4 & 6.10 & TO \\
wtsM2\_12 & 68.56 & TO & 85.35 & TO & 8.07 & TO & 4 & 7.46 & TO & 32 & {\bf 4.00} & {\bf TO} \\
wtsM2\_13 & 47.74 & TO & 54.04 & TO & 10.06 & TO & 4 & {\bf 4.24} & {\bf TO} & 8 & 5.72 & TO \\
wtsM2\_14 & 47.15 & TO & 46.24 & TO & 9.02 & TO & 32 & 6.34 & TO & 16 & {\bf 1.43} & {\bf TO} \\
wtsM2\_15 & 61.32 & TO & Inf & TO & 86.64 & TO & 4 & 8.81 & TO & 8 & {\bf 0.22} & {\bf TO} \\
wtsM2\_16 & 26.72 & TO & 47.77 & TO & 34.46 & TO & 8 & 34.46 & TO & 32 & {\bf 5.25} & {\bf TO} \\
lee1 & GOpt & 46.97 & GOpt & 145.55 & {\bf GOpt} & {\bf 13.01} & 8 & {\bf GOpt} & {\bf 13.01} & 8 & GOpt & 13.61 \\
lee2 & {\bf GOpt} & {\bf 60.43}& GOpt & 590.08 & 0.58 & TO & 16 & 0.47 & TO & 16 & 0.08 & TO \\
meyer4 & Inf & N/A & 80.40 & TO & GOpt & 18.85 & 4 & GOpt & 12.50 & 8 & {\bf GOpt} & {\bf 5.68} \\
meyer10 & 80.88 & TO & 239.70 & TO & GOpt & TO & 4 & GOpt & 452.53 & 8 & {\bf GOpt} & {\bf 133.47} \\
meyer15 & Inf & N/A & 2850.30 & TO & 0.59 & TO & 16 & 0.31 & TO & 4 & {\bf 0.10} & {\bf TO} \\
\bottomrule
    \end{tabular}
    \label{Tab:baron_3600}
\end{table}
\setlength\tabcolsep{6 pt} 

In this section we assess the empirical value of adaptive partitioning by presenting results without OBBT.
In Table \ref{Tab:baron_3600}, AMP (without OBBT) with CPLEX or Gurobi is compared with BARON. Columns two and three show the run times of Couenne and BARON, respectively, based on the 3600 second time limit. Column four shows the performance of AMP for $\Delta=8$.
Though $\Delta=8$ is not the ideal setting for every instance,
it is analogous to running BARON/Couenne with default parameters. Under the default AMP settings, AMP is faster than BARON and Couenne (with default settings) at finding the best lower bound in 23 out of 34 instances. { These results of AMP are the most fair to compare with untuned BARON and Couenne}

In column five (tuned $\Delta$ and CPLEX), the run times of AMP are much faster than BARON and Couenne in 24 out of 34 instances. Column six (tuned $\Delta$ and Gurobi) again indicates that the right choice of $\Delta$ speeds up the convergence of AMP drastically. More interestingly, on 21 out of 34 instances, the run times of AMP using Gurobi are substantially better than the run times using CPLEX (column 5). 

Table \ref{Tab:baron_3600} is summarized with a cumulative distribution plot in Figure \ref{Fig:baron_3600}. Clearly, Figure \ref{Fig:baron_3600}(a) indicates that AMP is better able to find solutions within a 0.4\% optimality gap (even without tuning).
Figure \ref{Fig:baron_3600}(b) provides evidence of the overall strength of AMP. Even when AMP is not the fastest approach, its run times are very similar to BARON. 
Overall, the performance of AMP is clearly better using Gurobi as the underlying MILP/MIQCQP solver. We did not perform comparative studies of BARON with Gurobi because it cannot currently integrate with Gurobi. 


\begin{figure}[h]
   \centering
  \subfigure[Comparison of best gap]{
  \includegraphics[width=0.72\textwidth]{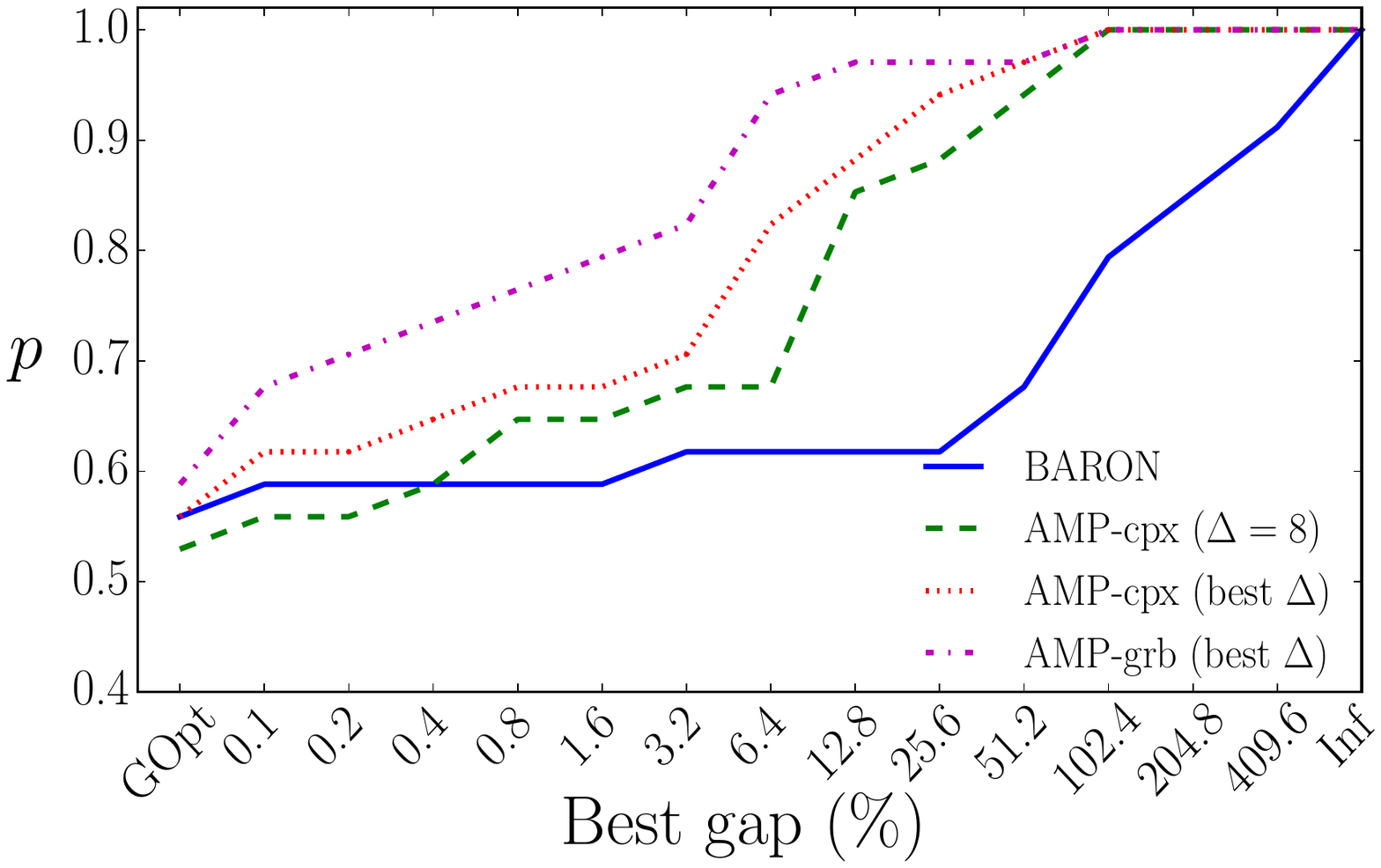}}
  \subfigure[Comparison of best run times]{
  \includegraphics[width=0.72\textwidth]{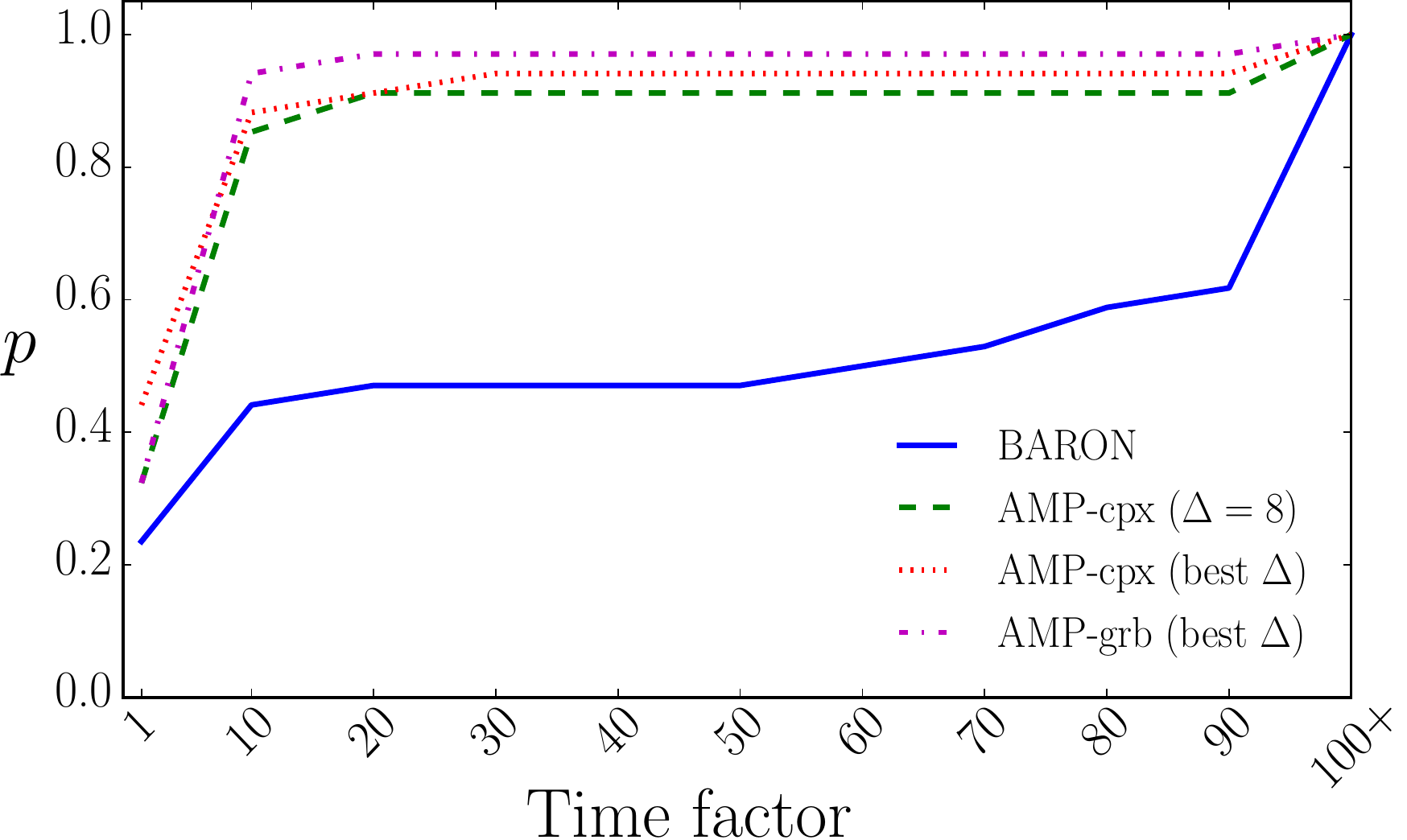}}
  \caption{Performance profiles of AMP (without OBBT) and BARON. In (a), the x axis plots the 
  optimality gap of the algorithms and the y axis plots fraction of instances. Plot (a) tracks the number of instances where an algorithm is able achieve the specified optimality gap. In (b), the x axis denotes the run time ratio of an algorithm with the best run time of any algorithm. The y axis denotes the fraction of instances. Plot (b) tracks the number of times an algorithm's run time is within a specified factor of the best run time of any algorithm.
  In both figures, higher is better. Overall, AMP performs better than BARON on a $p$ proportion of instances for most gaps and all run times.}
   \label{Fig:baron_3600}
\end{figure}

\subsection{Performance of AMP with OBBT}
\label{subsec:perf_withm_bnd}

We next discuss the performance of AMP when OBBT is added.

\subsubsection{Default Parameters of $\Delta$}
\label{subsubsec:without_param}
We first consider AMP with OBBT when AMP uses the default parameter of $\Delta=8$.
Table \ref{Tab:baron_ext}, compares AMP with BARON. 
Column two shows the run times of BARON based on a prescribed time limit. For comparison purposes with AMP, the time limit of BARON is calculated as the sum of 3600 seconds and the maximum of the run time of BT and PBT. For the purposes of this study, we did not specify a time limit on BT and PBT, though this could be added.

Column three shows the performance of BT-AMP when $\Delta=8$. 
While a constant $\Delta$ is not the ideal parameter for every instance, BT-AMP with CPLEX is still faster than BARON on 21 out of 32 instances. Similarly, BT-AMP with Gurobi is faster in 25 out of 32 instances. Once again, AMP with Gurobi has significant computational advantages over CPLEX. Instances
\emph{blend480}, \emph{blend721}, \emph{blend852}, and \emph{meyer10} demonstrated an order of magnitude improvement.
Column four of Table \ref{Tab:baron_ext} shows the results of PBT-AMP when $\Delta=10$.
Though PBT solves a more complicated, discrete optimization problem at every step of bound-tightening, surprisingly, the total time spent in bound-tightening was typically significantly smaller.
{This is seen in all instances prefixed with ``\emph{ex}''}, the \emph{util} instance, the \emph{eniplac} instance, the \emph{meanvarx} instance and a few \emph{blend} instances. {In general, using partition-based OBBT yields significant improvements in the overall run times and optimality gaps of AMP.}

The bound-tightening procedure also compares favorably with BARON. For example, consider problem \emph{blend852}. 
Though BARON implements a sophisticated bound-tightening approach that is based on primal and dual formulations, BARON times out with a 0.08\% gap. 
In contrast, BT-AMP with Gurobi converges to the global optimum in 434.7 seconds and PBT-AMP converges to the global optimum in 78.1 seconds (an order-of-magnitude improvement). Similar behaviour is observed on the remaining \emph{blend}, \emph{wts} and \emph{meyer} instances. 
Overall, AMP with Gurobi outperforms BARON on 24 out of 32 instances when a default choice of $\Delta$ is used.

Table \ref{Tab:baron_ext} is summarized with a cumulative distribution plot in Figure \ref{Fig:baron_ext}. Clearly, Figure \ref{Fig:baron_ext}(a) indicates that BT-AMP and PBT-AMP with Gurobi performs better than BARON even without tuning $\Delta$. In \ref{Fig:baron_ext}(a), AMP has a better profile when the optimality gap is $> 0.4\%$. In Figure \ref{Fig:baron_ext}(a), the performance improvement starts at $0.2\%$.
However, there is an increase in run times due to bound-tightening (Figure \ref{Fig:baron_ext}(b)), that allows to AMP to achieve this improvement.


\begin{remark}
\emph{meyer15}, a generalized pooling problem-based instance, is classified as a large-scale MINLP and is very hard for global optimization. The current best known gap for this instance is 0.1\% \cite{misener2011apogee,bussieck2003minlplib}. PBT-AMP with Gurobi has closed this problem by proving the \textit{global optimum} for the first time (\emph{943734.0215}--Table \ref{Tab:baron_ext}).   
\end{remark}

\begin{figure}[h]
   \centering
   \subfigure[Comparison of best gap]{
   \includegraphics[width=0.72\textwidth]{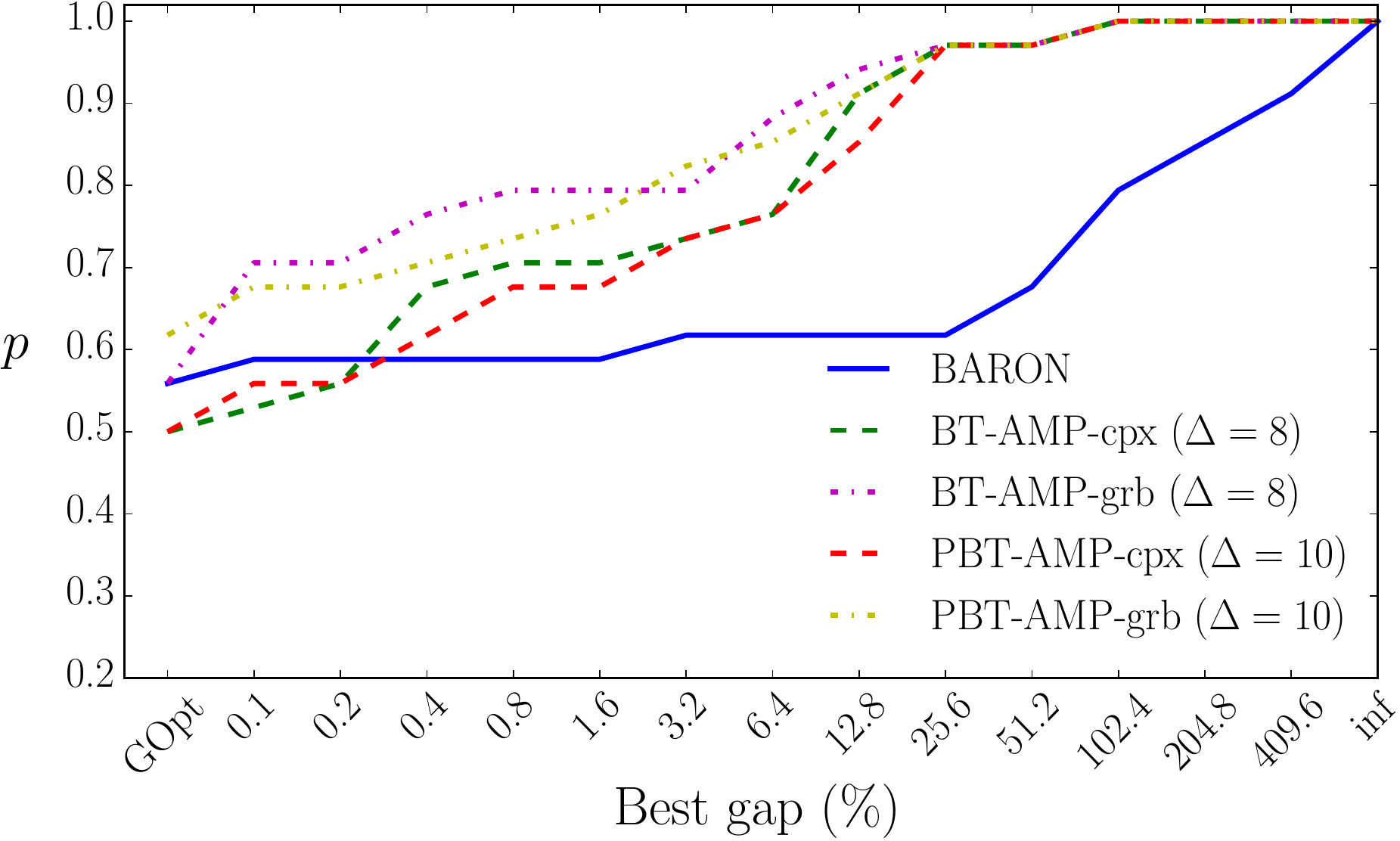}}
   \subfigure[Comparison of best run times]{
   \includegraphics[width=0.72\textwidth]{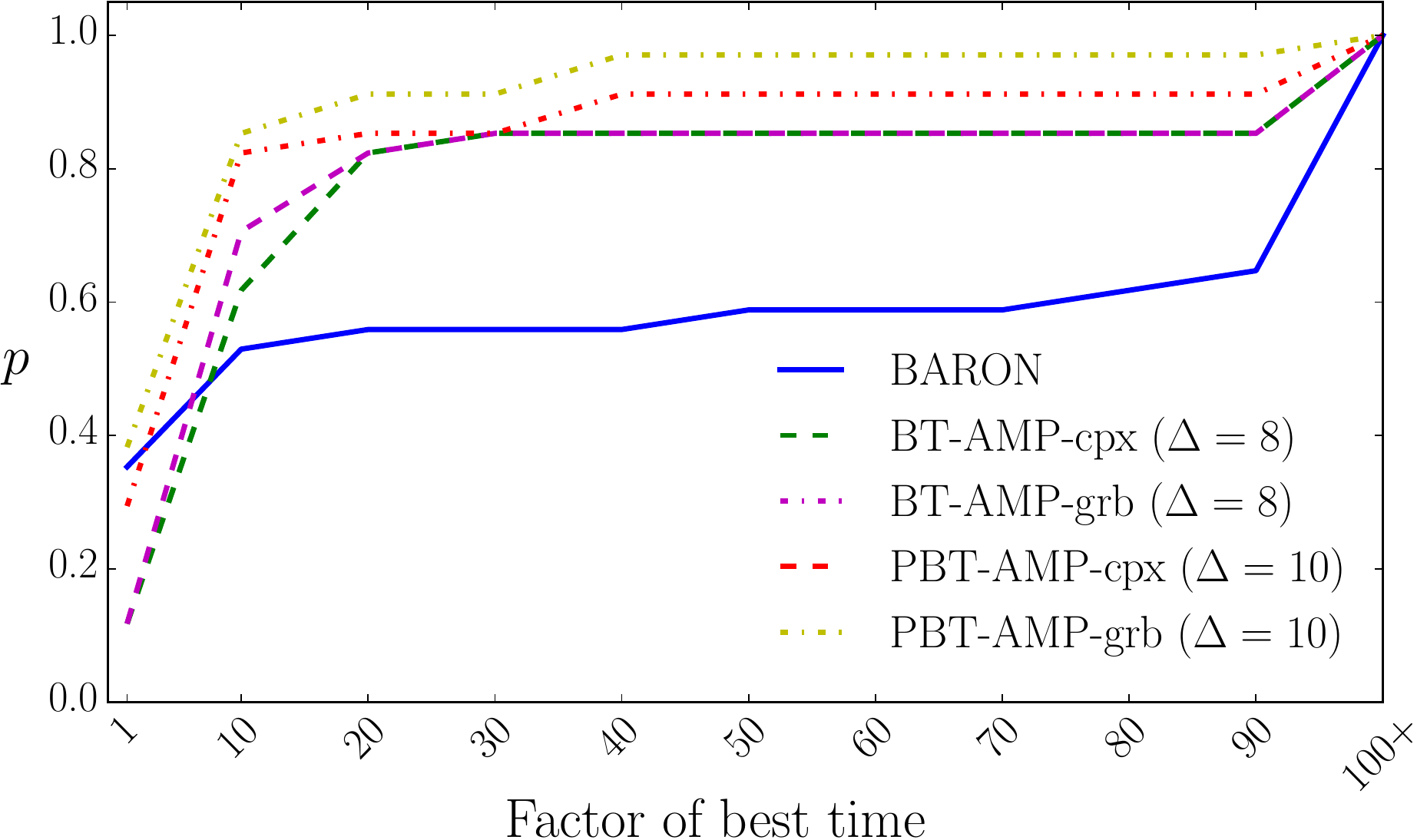}} 
   \caption{
   Performance profiles of AMP (with OBBT) and BARON.  In (a), the x axis plots the 
  optimality gap of the algorithms and the y axis plots fraction of instances. Plot (a) tracks the number of instances where an algorithm is able achieve the specified optimality gap. In (b), the x axis denotes the run time ratio of an algorithm with the best run time of any algorithm. The y axis denotes the fraction of instances. Plot (b) tracks the number of times an algorithm's run time is within a specified factor of the best run time of any algorithm.
  In both figures, higher is better. Overall, AMP performs better than BARON on $p$ proportion of instances within a factor of the best gap and with the best run times.}
   \label{Fig:baron_ext}
\end{figure}


\begin{landscape}
\begin{table}[htp]
\scriptsize
\caption{Performance summary of AMP with OBBT on all instances. 
    Here, we compare the run times of BARON, BT-AMP with $\Delta=8$ and PBT-AMP with $\Delta=10$. CPLEX and Gurobi are the underlying solvers for AMP. Values under ``Gap" and ``$T,T^+$" are in \% and seconds, respectively.``Inf" implies that the solver failed to provide a bound within the prescribed time limit.}
\begin{tabular}{l|rr|rrrrr|rrrrr}
\toprule
& \multicolumn{2}{|c|}{$\text{BARON}$} & \multicolumn{1}{c}{BT} & \multicolumn{2}{c}{AMP-cpx} & \multicolumn{2}{c|}{AMP-grb} & \multicolumn{1}{c}{PBT} & \multicolumn{2}{c}{AMP-cpx} & \multicolumn{2}{c}{AMP-grb} \\
\cmidrule(lr){2-3} \cmidrule(lr){4-4} \cmidrule(lr){5-6} \cmidrule(lr){7-8} \cmidrule(lr){9-9} \cmidrule(lr){10-11}
\cmidrule(lr){12-13}
Instances & Gap & $T$ & $T^{+}$ & Gap & $T$ & Gap & $T$ & $T^{+}$ & Gap & $T$ & Gap & $T$ \\
\midrule
fuel & GOpt & 0.03 & 0.01 & GOpt & 0.03 & GOpt & 0.03 & 0.01 & GOpt & 0.04 & GOpt & 0.03 \\
ex1223a & GOpt & 0.02 & 19.71 & GOpt & 0.01 & GOpt & 0.01 & 0.30 & GOpt & 0.01 & GOpt & 0.02 \\
ex1264 & GOpt & 1.44 & 12.47 & GOpt & 1.77 & GOpt & 1.32 & 0.72 & GOpt & 1.48 & GOpt & 1.24 \\
ex1265 & GOpt & 13.3 & 6.02 & GOpt & 0.25 & GOpt & 0.88 & 1.02 & GOpt & 0.26 & GOpt & 0.74 \\
ex1266 & GOpt & 10.81 & 13.75 & GOpt & 0.12 & GOpt & 0.06 & 1.30 & GOpt & 0.04 & GOpt & 0.06 \\
eniplac & GOpt & 207.37 & 16.04 & GOpt & 1.13 & GOpt & 1.36 & 3.34 & GOpt & 1.17 & GOpt & 1.59 \\
util & GOpt & 0.10 & 9.92 & GOpt & 0.17 & GOpt & 0.19 & 0.61 & GOpt & 0.14 & GOpt & 0.16 \\
meanvarx & GOpt & 0.05 & 20.69 & GOpt & 95.23 & GOpt & 59.33 & 3.53 & GOpt & 13.62 & GOpt & 13.31 \\
blend029 & GOpt & 2.46 & 15.05 & GOpt & 1.04 & GOpt & 1.56 & 0.80 & GOpt & 0.88 & GOpt & 0.95 \\
blend531 & GOpt & 111.79 & 44.67 & GOpt & 38.60 & GOpt & 22.56 & 477.94 & GOpt & 39.89 & GOpt & 20.12 \\
blend146 & 2.20 & TO & 30.95 & 24.92 & TO & 0.10 & TO & 26.66 & 24.98 & TO & 23.69 & TO \\
blend718 & 175.10 & TO & 28.76 & GOpt & 1332.66 & GOpt & 1335.93 & 20.8 & 14.65 & TO & GOpt & 868.14 \\
blend480 & GOpt & 326.95 & 137.62 & 0.21 & TO & GOpt & 108.93 & 1699.18 & 8.78 & TO & GOpt & 2466.17 \\
blend721 & GOpt & 548.9 & 29.44 & GOpt & 646.92 & GOpt & 181.88 & 23.92 & GOpt & 93.28 & GOpt & 112.91 \\
blend852 & 0.08 & TO & 41.73 & GOpt & 749.03 & GOpt & 392.99 & 29.40 & GOpt & 217.62 & GOpt & 48.79 \\
wtsM2\_05 & GOpt & 153.30 & 14.82 & 0.24 & TO & 0.02 & TO & 0.34 & 0.22 & TO & GOpt & 2875.17 \\
wtsM2\_06 & GOpt & 228.18 & 15.52 & 0.01 & TO & 0.01 & TO & 0.33 & 0.02 & TO & GOpt & 1957.59 \\
wtsM2\_07 & GOpt & 759.96 & 16.15 & 0.26 & TO & 0.30 & TO & 0.22 & 0.04 & TO & 1.59 & TO \\
wtsM2\_08 & 388.62 & TO & 21.54 & 14.37 & TO & 14.72 & TO & 0.90 & 20.08 & TO & 19.94 & TO \\
wtsM2\_09 & Inf & TO & 42.28 & 61.99 & TO & 64.53 & TO & 13.96 & 56.72 & TO & 55.85 & TO \\
wtsM2\_10 & 76.48 & TO & 15.73 & 0.10 & TO & 0.07 & TO & 0.32 & 0.22 & TO & 0.22 & TO \\
wtsM2\_11 & 107.56 & TO & 22.64 & 9.76 & TO & 3.74 & TO & 1.09 & 13.81 & TO & 9.18 & TO \\
wtsM2\_12 & 85.35 & TO & 39.10 & 11.90 & TO & 11.92 & TO & 3.44 & 6.03 & TO & 11.02 & TO \\
wtsM2\_13 & 54.04 & TO & 111.39 & 2.01 & TO & 3.95 & TO & 17.00 & 2.17 & TO & 2.10 & TO \\
wtsM2\_14 & 46.24 & TO & 19.09 & 6.64 & TO & 4.71 & TO & 1.10 & 1.93 & TO & 1.93 & TO \\
wtsM2\_15 & Inf & TO & 14.93 & 0.29 & TO & 0.50 & TO & 0.34 & 0.51 & TO & 0.48 & TO \\
wtsM2\_16 & 47.77 & TO & 22.54 & 8.76 & TO & 6.91 & TO & 1.18 & 9.40 & TO & 5.11 & TO \\
lee1 & GOpt & 145.55 & 13.28 & GOpt & 12.50 & GOpt & 13.55 & 0.22 & 10.00 & TO & 0.01 & TO \\
lee2 & GOpt & 590.08 & 14.92 & 0.58 & TO & 0.37 & TO & 5.35 & 0.43 & TO & 0.07 & TO \\
meyer4 & 80.40 & TO & 15.78 & GOpt & 4.22 & GOpt & 4.47 & 238.47 & GOpt & 14.76 & GOpt & 13.68 \\
meyer10 & 239.70 & TO & 44.68 & 9.74 & TO & GOpt & 2925.36 & 63.79 & GOpt & TO & GOpt & 1189.71 \\
meyer15 & 2556.37 & TO & 3877.09 & 3.44 & TO & 0.08 & TO & 17868.96 & GOpt & TO & GOpt & TO \\
\bottomrule
\end{tabular}
\label{Tab:baron_ext}
\end{table}
\end{landscape}

\subsubsection{Tuned $\Delta$ Parameter} 
In Table \ref{Tab:best_delta}, we show the results of AMP (with OBBT) when $\Delta$ is tuned for each problem instance and these results are compared with BARON. {For the purposes of this article, $\Delta$ is tuned by running AMP with $\Delta=\{4,8,10,16,32\}$ on each instance. We then choose the value of $\Delta$ that provides the best lower bound (global optimal in many cases) in the minimum amount of computation time. This tuned value of $\Delta$ is denoted by $\Delta^{*}$ in Table \ref{Tab:best_delta}. Developing adaptive and automatic tuning heuristic-algorithms for computing the value of $\Delta$ remains an open question and is a subject of future work.}
As the performance of AMP is consistently the strongest with Gurobi, we present those results.
Overall, this table shows the best results for AMP. Column two of this table shows the run times of BARON. Similar to table \ref{Tab:baron_ext}, the time limit for BARON is the sum of 3600 seconds and the maximum of the run times of the BT and PBT algorithms (no time limit on BT and PBT). Column three tabulates the performance of BT-AMP with Gurobi by choosing the best $\Delta$ parameter for each instance. Overall, BT-AMP and PBT-AMP performed better than BARON on 26 out of 32 instances. As discussed in detail in section \ref{subsubsec:without_param}, similar observations about the performance of our algorithms also hold for this table. Again, the performance of PBT, despite the use of discrete optimization, indicates that PBT is the strongest bound-tightening procedure.
On the large MINLP instance \emph{meyer15}, PBT has large computational overhead, but this overhead pays off when AMP converges to the global optimum in 538.56 seconds. As noted in the earlier remark, this was an open instance prior to this work.  

Table \ref{Tab:baron_ext} is summarized with the cumulative distribution plot shown in Figure \ref{Fig:best_delta}. Figure \ref{Fig:best_delta}(a) indicates that AMP, BT-AMP and PBT-AMP with Gurobi are better than BARON in finding the best lower bounds. Though the proportion of instances for which global optima are attained is not significantly different than BARON, the proportion of instances for which better lower bounds are found using AMP-based algorithms is larger. Also, the advantages of BT and PBT-based bound-tightening in AMP is evident from the fact that the proportion of instances that find global optimum is higher. 
As expected, Figure \ref{Fig:best_delta}(b) suggests that AMP with Gurobi is overall faster than BARON on the easiest instances, but on harder instances this speed is tempered by a degradation in solution quality.

\vspace{-0.65cm}
\begin{figure}[htp]
   \centering
   %
   \subfigure[Comparison on Best Gap \%]{
   \includegraphics[width=0.75\textwidth]{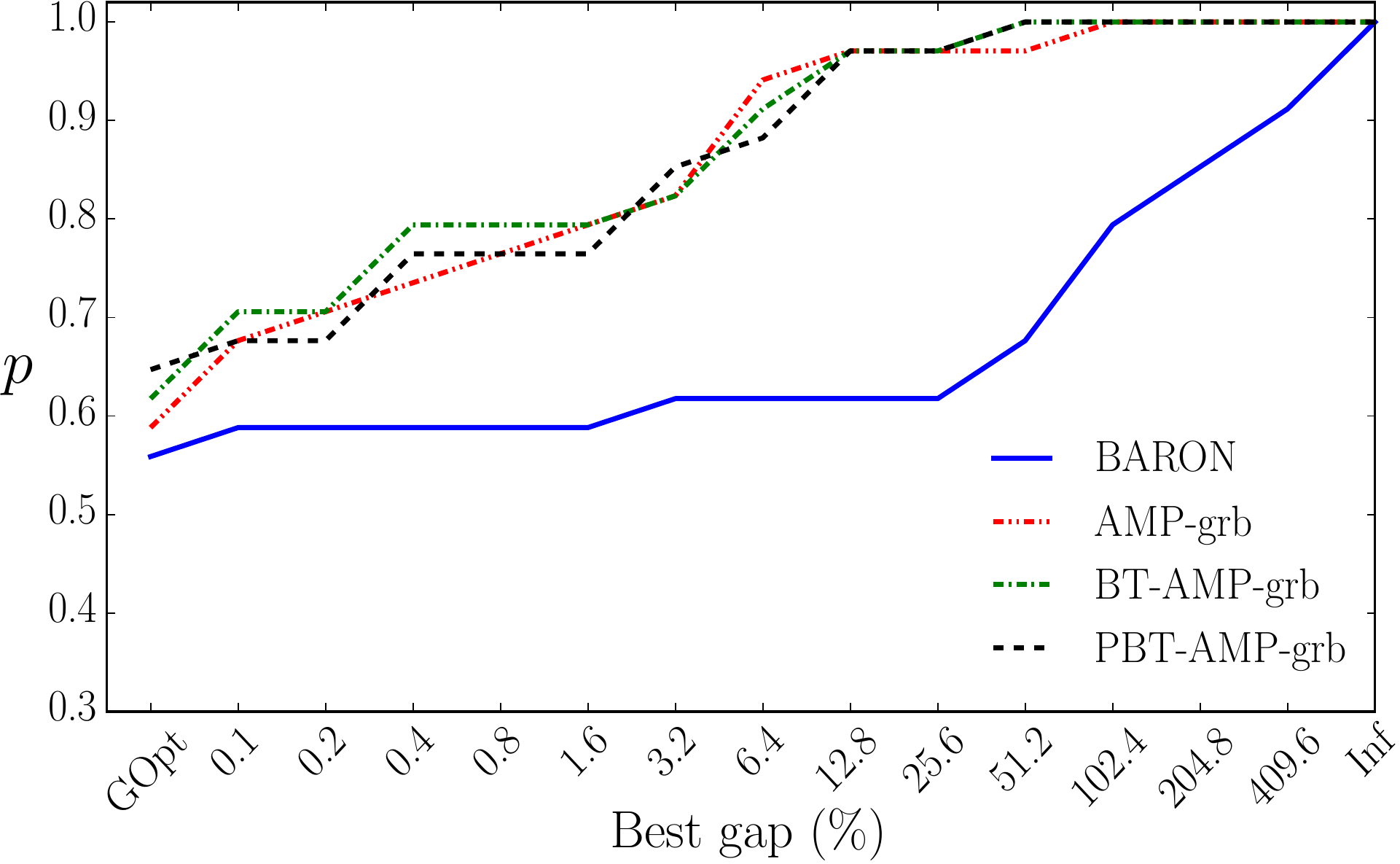}}
   \subfigure[Comparison on Best CPU Times]{
   \includegraphics[width=0.75\textwidth]{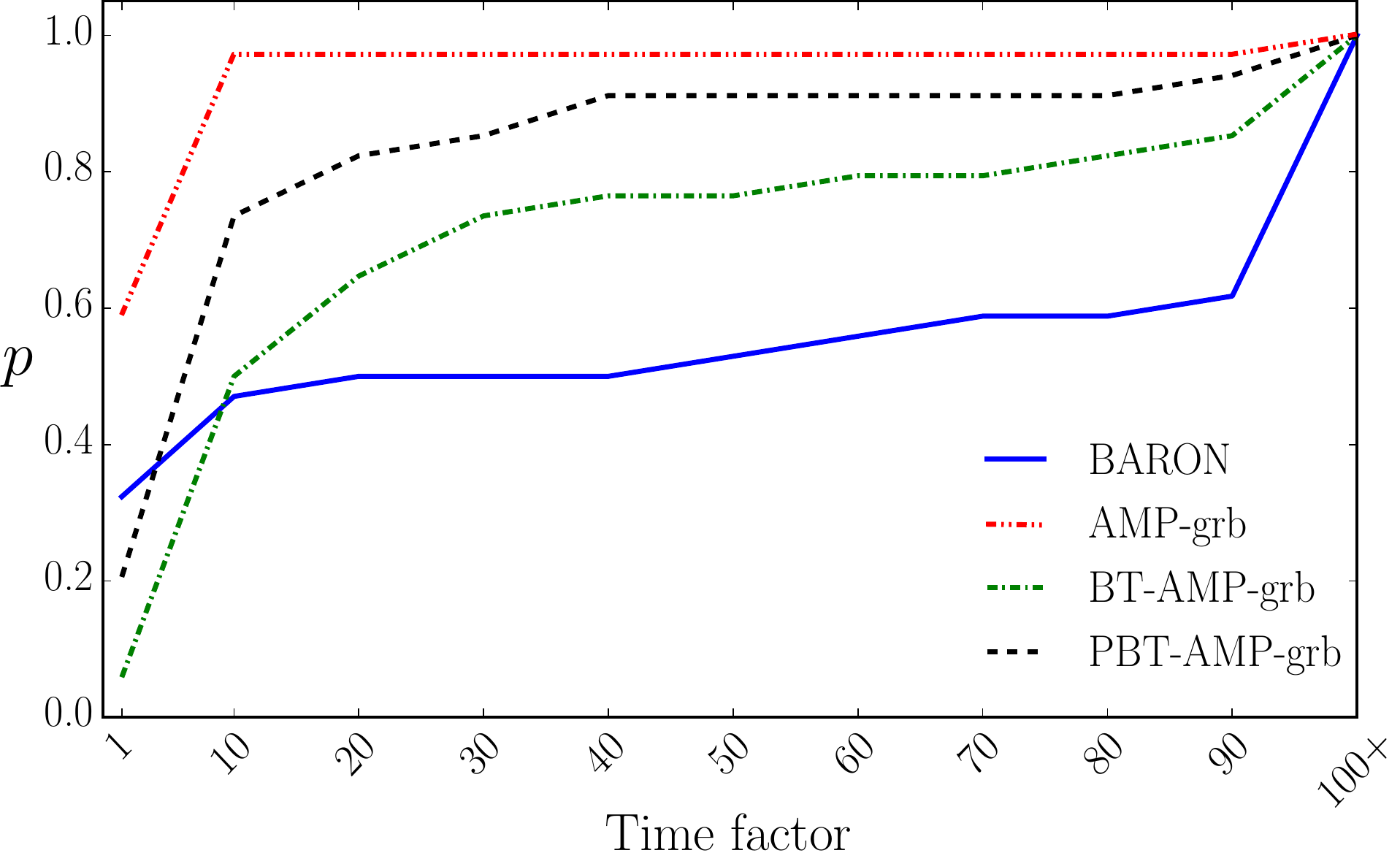}} 
   \caption{
      Performance profiles of AMP (with OBBT and tuned $\Delta$) and BARON.  In (a), the x axis plots the 
  optimality gap of the algorithms and the y axis plots fraction of instances. Plot (a) tracks the number of instances where an algorithm is able achieve the specified optimality gap. In (b), the x axis denotes the run time ratio of an algorithm with the best run time of any algorithm. The y axis denotes the fraction of instances. Plot (b) tracks the number of times an algorithm's run time is within a specified factor of the best run time of any algorithm.
  In both figures, higher is better. Overall, AMP-based algorithms perform better than BARON on $p$ proportion of instances within a factor of the best gap and with the best run times.}
   \label{Fig:best_delta}
\end{figure}



\begin{landscape}
\begin{table}[htp]
    \centering
    \scriptsize
    \caption{Summary of the performance of AMP with and without OBBT on all instances. 
    Here, we compare the run times of BARON, BT-AMP and PBT-AMP with tuned values of $\Delta$ for 
    each instance. Values under ``Gap'' and ``$T,T^+$'' are in \% and seconds, respectively.``Inf'' implies that the solver failed to provide a bound within the prescribed time limit. Gap values shown within parenthesis are evaluated using global optimum values instead of the best-found upper bound by AMP. For each instance, the bold face font represents best run time or the best optimality gap (if the solve times out)}
    \begin{tabular}{l|rr|rrrr|rrrr}
\toprule
& \multicolumn{2}{c}{BARON} & \multicolumn{4}{|c|}{BT-AMP-grb} & \multicolumn{4}{c}{PBT-AMP-grb} \\
\cmidrule(lr){2-3} \cmidrule(lr){4-7} \cmidrule(lr){8-11}
Instances & Gap & $T$ & $\Delta^{*}$ & Gap & $T^{+}$ & $T$ & $\Delta^{*}$ & Gap & $T^{+}$ & $T$ \\
\midrule
fuel & \textbf{GOpt} & \textbf{0.03} & 8 & GOpt & 0.01 & 0.04 & 8 & GOpt & 0.01 & 0.04 \\
ex1223a & \textbf{GOpt} & \textbf{0.02} & 8 & GOpt & 19.71 & 0.01 & 32 & GOpt & 0.31 & 0.01 \\
ex1264 & GOpt & 1.44 & 4 & GOpt & 12.47 & 0.59 & 4 & {\bf GOpt} & {\bf 0.84} & {\bf 0.49} \\
ex1265 & GOpt & 13.30 & 16 & GOpt & 6.02 & 0.45 & 16 & {\bf GOpt} & {\bf 0.92} & {\bf 0.46}\\
ex1266 & GOpt & 10.81 & 8 & GOpt & 13.75 & 0.06 & 4 & {\bf GOpt} & {\bf 1.18} & {\bf 0.09} \\
eniplac & GOpt & 207.37 & 32 & GOpt & 16.04 & 0.45 & 32 & {\bf GOpt} & {\bf 3.34} & {\bf 0.64} \\
util & {\bf GOpt} & {\bf 0.10} & 4 & GOpt & 9.92 & 0.08 & 4 & GOpt & 0.56 & 0.09 \\
meanvarx & {\bf GOpt} & {\bf 0.05} & 16 & GOpt & 20.69 & 21.17 & 16 & GOpt & 3.45 & 15.57 \\
blend029 & GOpt & 2.46 & 32 & GOpt & 15.05 & 0.92 & 16 & {\bf GOpt} & {\bf 1.09} & {\bf 1.33} \\
blend531 & GOpt & 111.79 & 8 & {\bf GOpt} & {\bf 44.67} & {\bf 22.56} & 32 & GOpt & 239.17 & 84.97 \\
blend146 & 2.20 & TO & 8 & {\bf 0.10} & {\bf 30.95} & {\bf TO} & 8 & 2.10 & 23.06 & TO \\
blend718 & 175.10 & TO & 32 & {\bf GOpt} & {\bf 28.76} & {\bf 889.28} & 8 & GOpt & 21.80 & 1101.56 \\
blend480 & GOpt & 326.95 & 8 & {\bf GOpt} & {\bf 137.62} & {\bf 108.93} & 16 & GOpt & 948.27 & 2185.03 \\
blend721 & GOpt & 548.90 & 16 & GOpt & 29.44 & 92.27 & 32 & {\bf GOpt} & {\bf 9.87} & {\bf 90.91} \\
blend852 & 0.08 & TO & 16 & GOpt & 41.73 & 323.86 & 16 & {\bf GOpt} & {\bf 14.16} & {\bf 323.31} \\
wtsM2\_05 & {\bf GOpt} & {\bf 153.30} & 16 & GOpt & 14.82 & 2482.73 & 16 & GOpt & 0.33 & 2483.36 \\
wtsM2\_06 & {\bf GOpt} & {\bf 228.18} & 16 & GOpt & 15.52 & 2058.92 & 16 & GOpt & 0.39 & 2057.20 \\
wtsM2\_07 & {\bf GOpt} & {\bf 759.96} & 8 & 0.30 & 16.15 & TO & 8 & 0.30 & 0.24 & TO \\
wtsM2\_08 & 388.62 & TO & 4 & {\bf 8.25} & {\bf 21.54} & {\bf TO} & 4 & 8.25 & 24.43 & TO \\
wtsM2\_09 & Inf & TO & 16 & 44.57 & 42.80 & TO & 16 & {\bf 43.66} & {\bf 288.88} & {\bf TO} \\
wtsM2\_10 & 76.48 & TO & 8 & 0.07 & 15.73 & TO & 8 & {\bf 0.06} & {\bf 0.39} & {\bf TO} \\
wtsM2\_11 & 107.56 & TO & 8 & 3.74 & 22.64 & TO & 8 & {\bf 2.39} & {\bf 1.11} & {\bf TO} \\
wtsM2\_12 & 85.35 & TO & 4 & (6.89) 7.29 & 39.10 & TO & 8 & (6.95) {\bf 7.34}  & {\bf 3.70} & {\bf TO} \\
wtsM2\_13 & 54.04 & TO & 8 & {\bf 3.95} & {\bf 111.39} & {\bf TO} & 4 & 6.92 & 130.31 & TO \\
wtsM2\_14 & 46.24 & TO & 32 & 2.71 & 19.09 & TO & 8 & {\bf 2.25} & {\bf 0.67} & {\bf TO} \\
wtsM2\_15 & Inf & TO & 32 & {\bf 0.20} & {\bf 14.93} & {\bf TO} & 16 & 0.28 & 0.35 & TO \\
wtsM2\_16 & 47.77 & TO & 4 & (2.61) {\bf 5.81}  & {\bf 22.54} & {\bf TO} & 4 & (3.04) 6.24  & 24.41 & TO \\
lee1 & GOpt & 145.55 & 8 & GOpt & 13.28 & 13.55 & 8 & {\bf GOpt} & {\bf 0.27} & {\bf 13.68} \\
lee2 & {\bf GOpt} & {\bf 590.08} & 16 & 0.36 & 14.92 & TO & 4 & 0.38 & 2.96 & TO \\
meyer4 & 80.40 & TO & 8 & {\bf GOpt} & {\bf 15.78} & {\bf 4.47} & 4 & GOpt & 77.61 & 8.45 \\
meyer10 & 239.70 & TO & 4 & {\bf GOpt} & {\bf 44.68} & {\bf 760.36} & 4 & GOpt & 34.67 & 775.74 \\
meyer15 & 2556.37 & TO & 4 & 0.02 & 3877.09 & TO & 4 & {\bf GOpt} & {\bf 19218.84} & {\bf 538.56} \\
\bottomrule
    \end{tabular}
    \label{Tab:best_delta}
\end{table}
\end{landscape}

\subsection{Sensitivity of MINLP structure}
\label{Subsec:when_AMP}
In Figure \ref{Fig:baronratio_3600}, we classify the MINLP instances to understand how problem structure influences the success of AMP.
There are various possible classification measures and we use the total number of variables that are part of multilinear terms. This is because our algorithm heavily depends on multi-variate partitioning on the nonlinear terms. Thus, it is likely that a measure like this influences the performance of AMP. Consider the following simple example that describes the measure clearly: Let $x_i, \ \forall i=1,\ldots,n$ be the variables in a problem with a linear objective and one nonlinear constraint,  $\left(\prod_{i=1}^{k}x_i + \prod_{i=2}^{k+1}x_i\right) \geqslant M$, such that $2 \leqslant k \leqslant n-1$. Then, the number of variables in mutlilinear terms is $k+1$. 

It is clear from the figure that both AMP and BT-AMP performs very well on instances that have large numbers of variables ($\geqslant$25) in the multilinear terms. We also observe that while executing OBBT incurs a computational overhead (ratio up to $\approx$16), there are many instances below the unit ratio value (blue dashed line).  Overall, these plots support the observation that increasing the number of variables in multilinear terms are indicator of success when executing AMP.

\begin{figure}[htp]
   \centering
   \subfigure[AMP]{
   \includegraphics[scale=0.3159]{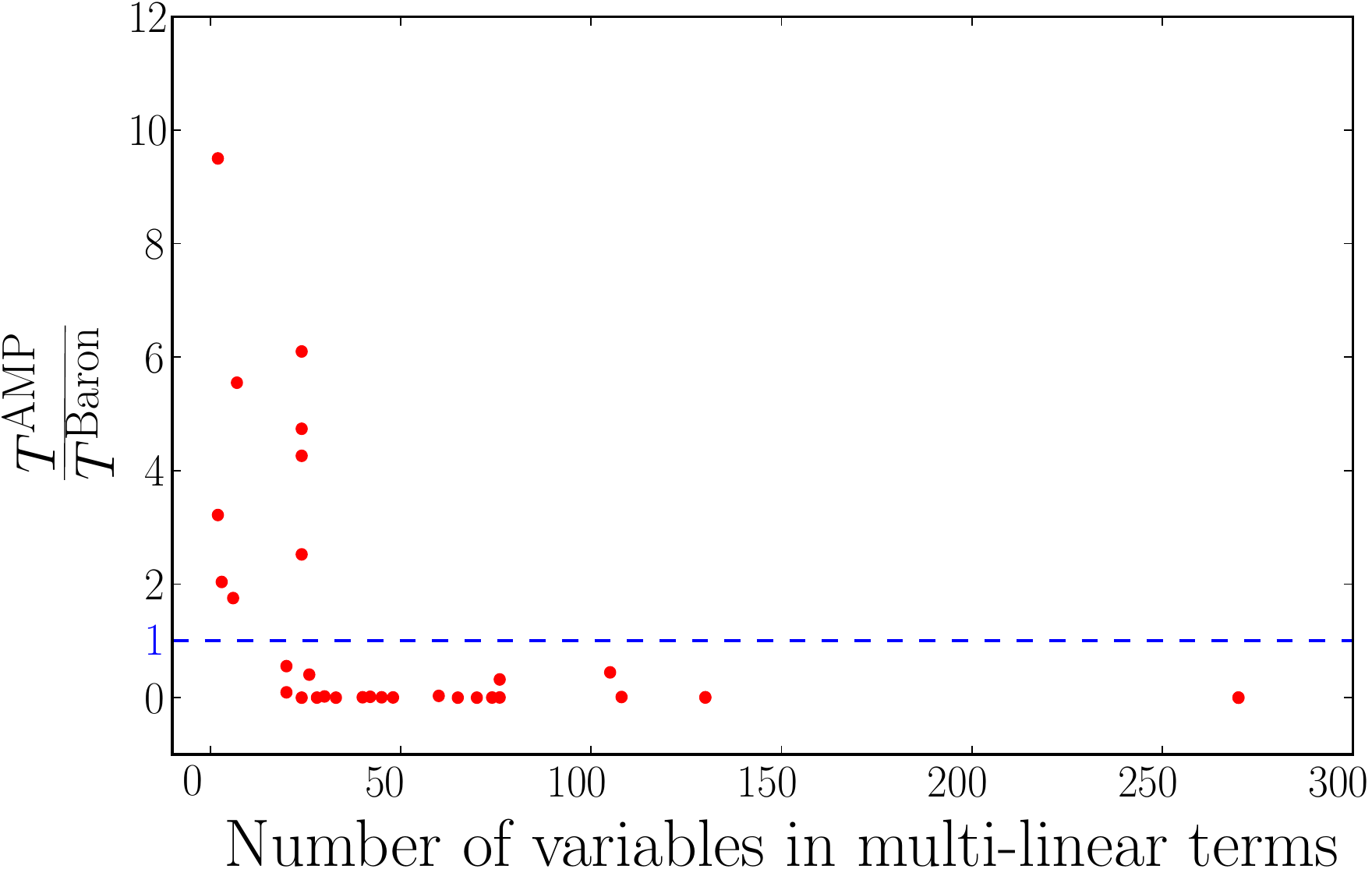}}
    \subfigure[BT-AMP]{
   \includegraphics[scale=0.3569]{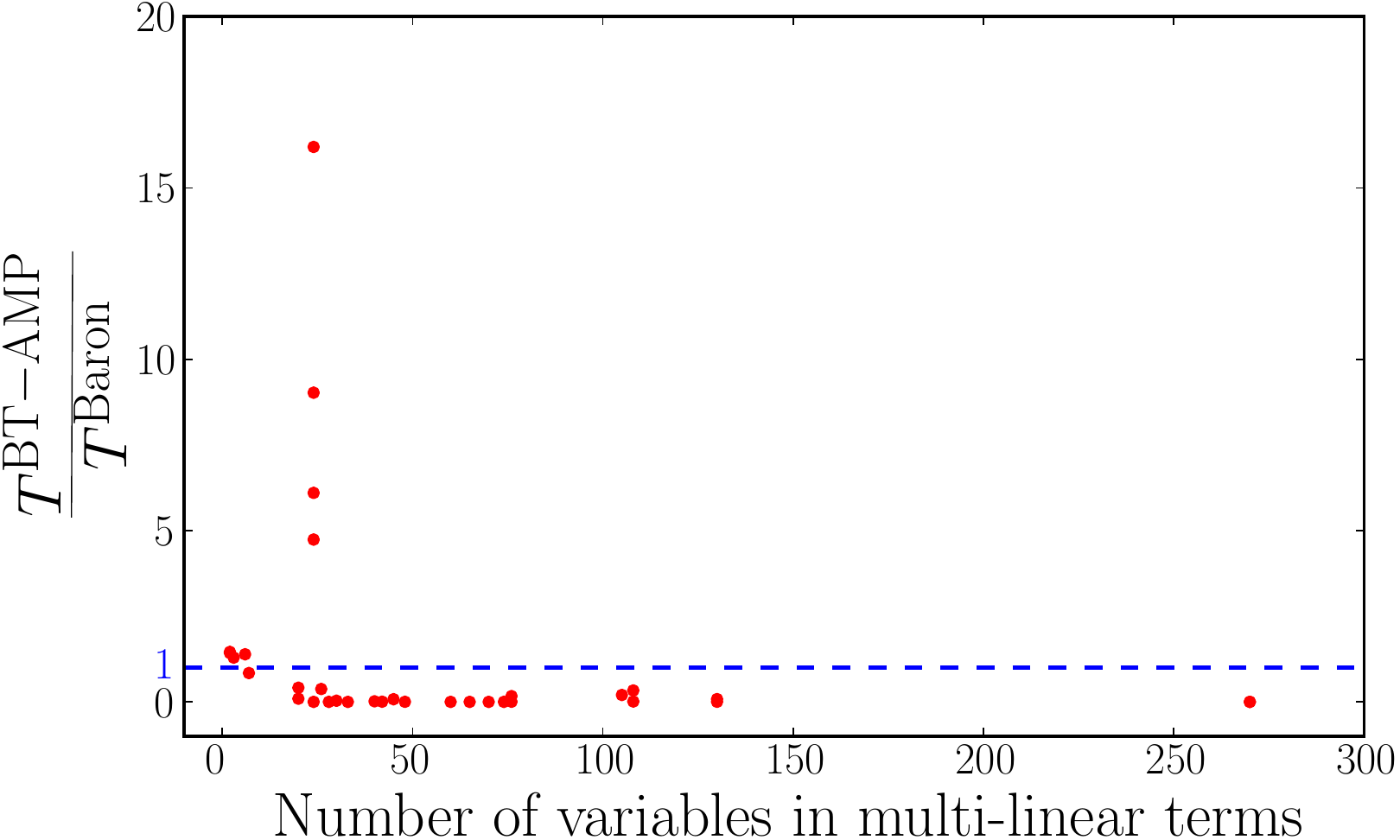}}
   \caption{Illustration of the ratio of run times of AMP and BT-AMP algorithms (with tuned parameters) to BARON. The y axis denotes the ratio and the x axis denotes the total number of variables in mutlilinear terms in a given MINLP instance. The blue dashed line indicates a ratio of 1. All red points correspond to a single instance. A point below the blue line indicates a ratio in favor of AMP.}
   \label{Fig:baronratio_3600}
\end{figure}


\subsection{OBBT Results for \textit{meyer15}}
One of the primary observations made in this paper is the importance of MIP-based sequential OBBT on medium-scale MINLPs.  However, for a given large-scale MINLP, one of the drawbacks of BT and PBT is that it solves MILPs to tighten the variable bounds. Solving MILPs can be time consuming, in particular on instances like \textit{meyer15}\footnote{\textit{meyer15} is a generalized pooling problem instance. These problems are typically considered hard (bilinear) MINLP for global optimization \cite{misener2011apogee,bussieck2003minlplib}}. Here, we focus on the run time issues associated with solving MILPs, suggest approaches for managing that run time and still get much of their benefits in bound-tightening.

Table \ref{Tab:meyer15_a} summarizes the run times of BT and PBT on \textit{meyer15} for various values of $\Delta$. As shown in the first row of this table, the run time varies drastically when the MILP is solved to optimality in every iteration of BT and PBT. To reduce this run time, we imposed a time limit on every $\min$- and $\max$-MILP of \textit{ten seconds}. Since early termination of MILP does not guarantee optimal primal-feasible solutions, incumbent solutions are not valid for bound-tightening. Instead, we use the best lower bound maintained by the solver. This ensures the validity of the tightened bounds. As shown in the second row of Table \ref{Tab:meyer15_a}, the run times of PBT are drastically reduced. 

Table \ref{Tab:meyer15_b} summarizes the results of AMP based on the tightened variable bounds presented in Table \ref{Tab:meyer15_a}. Interestingly, on \textit{meyer15}, we observed that AMP converges to near optimal solutions (sometimes, even better than solving full MILPs) when the time limit on MILP solvers is imposed. This is an important feature for tuning the time spent tightening bounds vs bound solution quality.
Further, this intriguing result suggests further study of MIP-based relaxations for OBBT of MINLPs, which we delegate for future work. 








\begin{table}[htp]
    \centering
    \caption{BT and PBT run times on \textit{meyer15} with and without a limit on the run time of every MILP solved during the bound-tightening phase. Here, a 10-second time limit was used. The imposition of this limit reduces the total run time of PBT to close to BT.}
    \begin{tabular}{lrrrrrr}
\toprule
& \multicolumn{1}{c}{BT} & \multicolumn{5}{c}{PBT} \\
\cmidrule(lr){2-2} \cmidrule(lr){3-7}
 &  & $\Delta=4$ & $\Delta=8$ & $\Delta=10$ & $\Delta=16$ & $\Delta=32$ \\
\midrule
Without time limit & 3877 & 19218  & 32130      & 29586       & 11705       & 17868       \\
With time limit & 3830 & 3914   & 4059       & 4014        & 3635        & 3621        \\
\bottomrule
    \end{tabular}
    \label{Tab:meyer15_a}
\end{table}

\setlength\tabcolsep{3 pt}

\begin{table}[htp]
    \centering
    \caption{AMP run times on \textit{meyer15} instance preceded by bound-tightening with (BT-lim, PBT-lim) and without limit (BT, PBT) on run time per iteration. Values under ``Gap" and ``$T$" are in \% and seconds, respectively. Bold font represents the best result in each row.}
    \begin{tabular}{lrrrrrrrrrr}
\toprule
& \multicolumn{2}{c}{$\Delta=4$} & \multicolumn{2}{c}{$\Delta=8$} & \multicolumn{2}{c}{$\Delta=10$} & \multicolumn{2}{c}{$\Delta=16$} & \multicolumn{2}{c}{$\Delta=32$}  \\
\cmidrule(lr){2-3} \cmidrule(lr){4-5} \cmidrule(lr){6-7} \cmidrule(lr){8-9} \cmidrule(lr){10-11}
              & Gap           & $T$           & Gap  & $T$    & Gap           & $T$         & Gap           & $T$           & Gap           & $T$           \\
\midrule
After BT      & 0.02          & TO            & 0.08 & TO     & 0.12 & TO & 0.31          & TO            & 0.15          & TO \\
After BT-lim  & 0.37          & TO            & 0.14 & TO     & 0.17 & TO & 0.69 & TO   & 0.90 & TO   \\
After PBT     & GOpt & 536.56 & GOpt& 1600    & GOpt          & TO   & 0.20  & TO            & 0.15          & TO \\
After PBT-lim & 0.37 & TO     & 0.90 & TO     & 0.19          & TO   & 0.70  & TO        & 0.90          & TO \\
\bottomrule
    \end{tabular}
    \label{Tab:meyer15_b}
\end{table}
\setlength\tabcolsep{6 pt} 

\begin{table}
\scriptsize
\caption{Overall structure of the MINLP problems. The first column describes the name of the problem instance. The second column cites the source of the problem. The third column shows the optimal solution. The fourth, fifth and sixth columns show the number of constraints, binary variables, and continuous variables, respectively. The seventh column indicates the partitioned continuous variables. ``ALL'' refers to all variables in mutlilinear terms and ``VC'' refers to variables in the minimum vertex cover as described in \cite{boukouvala2016global}. The final column shows the number of mutlilinear terms.}
\begin{tabular}{lcrrrrrr}
\toprule
Instance  & Ref. & GOpt  & \#Cons &\#BVars& \#CVars         & \#CVars-P          & \#ML  \\
          &           &       &        &       &            &                   &       \\
\hline
NLP1	  & \cite{nagarajan2016tightening}  & 7049.248      & 14	&	0	& 8     &ALL    &	5 	\\
fuel	  & \cite{bussieck2003minlplib}     & 8566.119      & 15	&	3	& 12    &VC     &	3 	\\
ex1223a	  & \cite{bussieck2003minlplib}     & 4.580	        & 9	    &	4	& 3     &VC     &	3 	\\
ex1264	  & \cite{bussieck2003minlplib}     & 8.6	        & 55	&	68	& 20    &VC     &	16 	\\
ex1265	  & \cite{bussieck2003minlplib}     & 10.3	        & 74	&	100	& 30    &VC     &	25 	\\
ex1266	  & \cite{bussieck2003minlplib}     & 16.3	        & 95	&	138	& 42    &VC     &	36 	\\
eniplac	  & \cite{bussieck2003minlplib}     & -132117.083	& 189	&	24	& 117   &VC     &	66	\\
util	  & \cite{bussieck2003minlplib}     & 999.578	    & 167	&	28	& 117   &ALL    &	5	\\
meanvarx  &	\cite{bussieck2003minlplib}     & 14.369	    & 44	&	14	& 21    &VC     &	28	\\
blend029  &	\cite{bussieck2003minlplib}     & 13.359	    & 213	&	36	& 66    &VC     &	28	\\
blend531  & \cite{bussieck2003minlplib}     & 20.039	    & 736	&	104	& 168   &VC     &	146	\\
blend146  & \cite{bussieck2003minlplib}     & 45.297	    & 624	&	87	& 135   &VC     &	104	\\
blend718  & \cite{bussieck2003minlplib}     & 7.394	        & 606	&	87	& 135   &VC     &	100	\\
blend480  &	\cite{bussieck2003minlplib}     & 9.227	        & 884	&	124	& 188   &VC     &	152	\\
blend721  & \cite{bussieck2003minlplib}     & 13.5268	    & 627	&	87	& 135   &VC	    &	104	\\
blend852  & \cite{bussieck2003minlplib}	    & 53.9626	    & 2412	&	120	& 184   &VC	    &	152	\\
wtsM2\_05 &	\cite{misener2013glomiqo}       & 229.7008	    & 152	&	0	& 134   &VC	    &	48	\\
wtsM2\_06 &	\cite{misener2013glomiqo}       & 173.4784	    & 152	&	0	& 134   &VC	    &	48	\\
wtsM2\_07 &	\cite{misener2013glomiqo}       & 80.77892	    & 152	&	0	& 134   &VC	    &	48	\\
wtsM2\_08 &	\cite{misener2013glomiqo}       & 109.4014	    & 335	&	0	& 279   &VC	    &	84	\\
wtsM2\_09 &	\cite{misener2013glomiqo}       & 124.4421	    & 573	&	0	& 517   &ALL	    &	210	\\
wtsM2\_10 &	\cite{misener2013glomiqo}       & 586.68	    & 138	&	0	& 156   &VC	    &	60	\\
wtsM2\_11 &	\cite{misener2013glomiqo}       & 2127.115	    & 252	&	0	& 304   &VC	    &	112	\\
wtsM2\_12 &	\cite{misener2013glomiqo}       & 1201.038	    & 408	&	0	& 517   &VC	    &	220	\\
wtsM2\_13 &	\cite{misener2013glomiqo}       & 1564.958	    & 783	&	0	& 1040  &VC	    &	480	\\
wtsM2\_14 &	\cite{misener2013glomiqo}       & 513.009	    & 205	&	0	& 209   &VC	    &	90	\\
wtsM2\_15 &	\cite{misener2013glomiqo}       & 2446.429	    & 152	&	0	& 134   &VC	    &	48	\\
wtsM2\_16 &	\cite{misener2013glomiqo}       & 1358.663	    & 234	&	0	& 244   &VC	    &	126	\\
lee1	  & \cite{MINLP123}                 & -4640.0824	& 82	&	9	& 40    &VC	    &	24	\\
lee2	  & \cite{MINLP123}                 & -3849.2654	& 92	&	9	& 44    &VC	    &	36	\\
meyer4	  & \cite{MINLP123}                 & 1086187.137	& 118	&	55	& 63    &VC	    &	48	\\
meyer10	  & \cite{MINLP123}                 & 1086187.137	& 423	&	187	& 207   &VC	    &	300	\\
meyer15	  & \cite{MINLP123}                 & 943734.0215	& 768	&	352	& 382   &VC	    &	675	\\
\bottomrule 
\end{tabular}
\label{tab:data}
\end{table}%


\section{Conclusions}
\label{sec:conclusions}

In this work, we developed an approach for adaptively partitioning nonconvex functions in MINLPs. 
We show that an adaptive partitioning of the domains of variables outperforms uniform partitioning, though the latter exhibits better optimality gaps in the first few iterations of the lower-bounding algorithm. We also show that bound-tightening techniques can be applied in conjunction with adaptive partitioning to improve convergence dramatically. We then use combinations of these techniques to develop an algorithm for solving MINLPs to global optimality.
Our numerical experiments on MINLPs with polynomials suggests that this is a very strong approach with an advantage of having very few tuning parameters in contrast to the existing methods. 

We have seen that using a well-designed MIP-based method with adaptive partitioning schemes is an attractive way of tackling MINLPs. With an apriori fixed tolerance, we get global optimum solutions by utilizing the well-developed state-of-the-art MIP solvers. However,
though AMP is relatively faster than the available global solvers, we observed that the computation times remain large for MINLPs with large number of nonconvex terms, thus motivating a multitude of directions for further developments. First, it will be 
important to consider existing classical nonlinear programming techniques, such as dual-based bound-contraction, 
partition elimination within the branch-and-bound search tree, bound-tightening at sub nodes \cite{mouret2009tightening}, and constraint propagation methods \cite{belotti2013bound}, which can tremendously speed-up our algorithm. 
Second, providing apriori guarantees on the size of the added partitions  ($\Delta$) that leads to faster tightening of the relaxations. This will support automatic tuning of $\Delta$ from within AMP. Third, recent developments on generating tight convex hull-reformulation-based cutting planes 
for solving convex generalized disjunctive programs will be very effective for attaining faster convergence to global optimum \cite{trespalacios2016cutting}. Finally, extensions of our methods from polynomial to general nonconvex functions (including fractional exponents, transcendental functions and disjunctions of nonconvex functions) will be another direction that will have relevance to numerous practical applications.

\begin{acknowledgements}
The work was funded by the Center for Nonlinear Studies (CNLS) at LANL and the LANL's directed research and development project "POD: A Polyhedral Outer-approximation, Dynamic-discretization optimization solver". Work was carried out under the auspices of the U.S. DOE under Contract No. DE-AC52-06NA25396. 
\end{acknowledgements}

\bibliographystyle{spmpsci}      
\bibliography{references.bib}

\begin{thebibliography}{10}
\providecommand{\url}[1]{{#1}}
\providecommand{\urlprefix}{URL }
\expandafter\ifx\csname urlstyle\endcsname\relax
  \providecommand{\doi}[1]{DOI~\discretionary{}{}{}#1}\else
  \providecommand{\doi}{DOI~\discretionary{}{}{}\begingroup
  \urlstyle{rm}\Url}\fi

\bibitem{achterberg2009scip}
Achterberg, T.: Scip: solving constraint integer programs.
\newblock Mathematical Programming Computation \textbf{1}(1), 1--41 (2009)

\bibitem{al1983jointly}
Al-Khayyal, F.A., Falk, J.E.: Jointly constrained biconvex programming.
\newblock Mathematics of Operations Research \textbf{8}(2), 273--286 (1983)

\bibitem{belotti2013bound}
Belotti, P.: Bound reduction using pairs of linear inequalities.
\newblock Journal of Global Optimization \textbf{56}(3), 787--819 (2013)

\bibitem{belotti2012feasibility}
Belotti, P., Cafieri, S., Lee, J., Liberti, L.: On feasibility based bounds
  tightening  (2012).
\newblock
  \urlprefix\url{https://hal.archives-ouvertes.fr/file/index/docid/935464/filename/377.pdf}

\bibitem{belotti2009branching}
Belotti, P., Lee, J., Liberti, L., Margot, F., W{\"a}chter, A.: Branching and
  bounds tightening techniques for non-convex minlp.
\newblock Optimization Methods \& Software \textbf{24}(4-5), 597--634 (2009)

\bibitem{bent2017polyhedral}
Bent, R., Nagarajan, H., Sundar, K., Wang, S., Hijazi, H.: A polyhedral
  outer-approximation, dynamic-discretization optimization solver, 0.1.0.
\newblock Tech. rep., Los Alamos National Laboratory, Los Alamos, NM, USA
  (2017).
\newblock \urlprefix\url{{https://github.com/lanl-ansi/Alpine.jl}}

\bibitem{bergamini2008improved}
Bergamini, M.L., Grossmann, I., Scenna, N., Aguirre, P.: An improved piecewise
  outer-approximation algorithm for the global optimization of {MINLP} models
  involving concave and bilinear terms.
\newblock Computers \& Chemical Engineering \textbf{32}(3), 477--493 (2008)

\bibitem{boukouvala2016global}
Boukouvala, F., Misener, R., Floudas, C.A.: Global optimization advances in
  mixed-integer nonlinear programming, minlp, and constrained derivative-free
  optimization, cdfo.
\newblock European Journal of Operational Research \textbf{252}(3), 701--727
  (2016)

\bibitem{bussieck2003minlplib}
Bussieck, M.R., Drud, A.S., Meeraus, A.: {MINLPL}ib—a collection of test
  models for mixed-integer nonlinear programming.
\newblock INFORMS Journal on Computing \textbf{15}(1), 114--119 (2003)

\bibitem{cafieri2010convex}
Cafieri, S., Lee, J., Liberti, L.: On convex relaxations of quadrilinear terms.
\newblock Journal of Global Optimization \textbf{47}(4), 661--685 (2010)

\bibitem{castro2015tightening}
Castro, P.M.: Tightening piecewise {McC}ormick relaxations for bilinear
  problems.
\newblock Computers \& Chemical Engineering \textbf{72}, 300--311 (2015)

\bibitem{castro2015normalized}
Castro, P.M.: Normalized multiparametric disaggregation: an efficient
  relaxation for mixed-integer bilinear problems.
\newblock Journal of Global Optimization \textbf{64}(4), 765--784 (2016)

\bibitem{coffrin2015strengthening}
Coffrin, C., Hijazi, H.L., Van~Hentenryck, P.: Strengthening convex relaxations
  with bound tightening for power network optimization.
\newblock In: Principles and Practice of Constraint Programming, pp. 39--57.
  Springer (2015)

\bibitem{dunning2017jump}
Dunning, I., Huchette, J., Lubin, M.: Jump: A modeling language for
  mathematical optimization.
\newblock SIAM Review \textbf{59}(2), 295--320 (2017)

\bibitem{DAmbrosio2010}
D’Ambrosio, C., Lodi, A., Martello, S.: Piecewise linear approximation of
  functions of two variables in milp models.
\newblock Operations Research Letters \textbf{38}(1), 39--46 (2010)

\bibitem{faria2011novel}
Faria, D.C., Bagajewicz, M.J.: Novel bound contraction procedure for global
  optimization of bilinear {MINLP} problems with applications to water
  management problems.
\newblock Computers \& chemical engineering \textbf{35}(3), 446--455 (2011)

\bibitem{faria2012new}
Faria, D.C., Bagajewicz, M.J.: A new approach for global optimization of a
  class of minlp problems with applications to water management and pooling
  problems.
\newblock AIChE Journal \textbf{58}(8), 2320--2335 (2012)

\bibitem{grossmann2013systematic}
Grossmann, I.E., Trespalacios, F.: Systematic modeling of discrete-continuous
  optimization models through generalized disjunctive programming.
\newblock AIChE Journal \textbf{59}(9), 3276--3295 (2013)

\bibitem{hasan2010piecewise}
Hasan, M., Karimi, I.: Piecewise linear relaxation of bilinear programs using
  bivariate partitioning.
\newblock AIChE journal \textbf{56}(7), 1880--1893 (2010)

\bibitem{hijazi2016}
Hijazi, H., Coffrin, C., Van~Hentenryck, P.: Convex quadratic relaxations for
  mixed-integer nonlinear programs in power systems.
\newblock Mathematical Programming Computation \textbf{9}(3), 321--367 (2017)

\bibitem{hock1980test}
Hock, W., Schittkowski, K.: Test examples for nonlinear programming codes.
\newblock Journal of Optimization Theory and Applications \textbf{30}(1),
  127--129 (1980)

\bibitem{horst2013handbook}
Horst, R., Pardalos, P.M.: Handbook of global optimization, vol.~2.
\newblock Springer Science \& Business Media (2013)

\bibitem{horst2013global}
Horst, R., Tuy, H.: Global optimization: Deterministic approaches.
\newblock Springer Science \& Business Media (2013)

\bibitem{karuppiah2006global}
Karuppiah, R., Grossmann, I.E.: Global optimization for the synthesis of
  integrated water systems in chemical processes.
\newblock Computers \& Chemical Engineering \textbf{30}(4), 650--673 (2006)

\bibitem{kocuk2016strong}
Kocuk, B., Dey, S.S., Sun, X.A.: Strong socp relaxations for the optimal power
  flow problem.
\newblock Operations Research \textbf{64}(6), 1177--1196 (2016)

\bibitem{kolodziej2013discretization}
Kolodziej, S.P., Grossmann, I.E., Furman, K.C., Sawaya, N.W.: A
  discretization-based approach for the optimization of the multiperiod blend
  scheduling problem.
\newblock Computers \& Chemical Engineering \textbf{53}, 122--142 (2013)

\bibitem{Li2013b}
Li, H.L., Huang, Y.H., Fang, S.C.: A logarithmic method for reducing binary
  variables and inequality constraints in solving task assignment problems.
\newblock INFORMS Journal on Computing \textbf{25}(4), 643--653 (2012)

\bibitem{liberti2008branch}
Liberti, L., Lavor, C., Maculan, N.: A branch-and-prune algorithm for the
  molecular distance geometry problem.
\newblock International Transactions in Operational Research \textbf{15}(1),
  1--17 (2008)

\bibitem{LuPSCC2017}
Lu, M., Nagarajan, H., Bent, R., Eksioglu, S., Mason, S.: Tight piecewise
  convex relaxations for global optimization of optimal power flow.
\newblock In: Power Systems Computation Conference (PSCC), pp. 1--7. IEEE
  (2018)

\bibitem{lu_optimal_tps2017}
Lu, M., Nagarajan, H., Yamangil, E., Bent, R., Backhaus, S., Barnes, A.:
  Optimal transmission line switching under geomagnetic disturbances.
\newblock IEEE Transactions on Power Systems \textbf{33}(3), 2539--2550 (2018).
\newblock \doi{10.1109/TPWRS.2017.2761178}

\bibitem{luedtke2012some}
Luedtke, J., Namazifar, M., Linderoth, J.: Some results on the strength of
  relaxations of multilinear functions.
\newblock Mathematical programming \textbf{136}(2), 325--351 (2012)

\bibitem{mccormick1976computability}
McCormick, G.P.: Computability of global solutions to factorable nonconvex
  programs: Part i—convex underestimating problems.
\newblock Mathematical programming \textbf{10}(1), 147--175 (1976)

\bibitem{meyer2006global}
Meyer, C.A., Floudas, C.A.: Global optimization of a combinatorially complex
  generalized pooling problem.
\newblock AIChE journal \textbf{52}(3), 1027--1037 (2006)

\bibitem{MINLP123}
Misener, R., Floudas, C.: Generalized pooling problem (2011).
\newblock Available from Cyber-Infrastructure for {MINLP} [{\tt www.minlp.org},
  a collaboration of Carnegie Mellon University and IBM Research] at: {\tt
  www.minlp.org/library/problem/index.php?i=123}

\bibitem{misener2013glomiqo}
Misener, R., Floudas, C.A.: Glomiqo: Global mixed-integer quadratic optimizer.
\newblock Journal of Global Optimization \textbf{57}(1), 3--50 (2013)

\bibitem{misener2011apogee}
Misener, R., Thompson, J.P., Floudas, C.A.: Apogee: Global optimization of
  standard, generalized, and extended pooling problems via linear and
  logarithmic partitioning schemes.
\newblock Computers \& Chemical Engineering \textbf{35}(5), 876--892 (2011)

\bibitem{mouret2009tightening}
Mouret, S., Grossmann, I.E., Pestiaux, P.: Tightening the linear relaxation of
  a mixed integer nonlinear program using constraint programming.
\newblock In: Integration of AI and OR Techniques in Constraint Programming for
  Combinatorial Optimization Problems, pp. 208--222. Springer (2009)

\bibitem{nagarajan2016tightening}
Nagarajan, H., Lu, M., Yamangil, E., Bent, R.: Tightening {McC}ormick
  relaxations for nonlinear programs via dynamic multivariate partitioning.
\newblock In: International Conference on Principles and Practice of Constraint
  Programming, pp. 369--387. Springer (2016)

\bibitem{nagarajan2017r2r}
Nagarajan, H., Pagilla, P., Darbha, S., Bent, R., Khargonekar, P.: Optimal
  configurations to minimize disturbance propagation in manufacturing networks.
\newblock In: American Control Conference (ACC), 2017, pp. 2213--2218. IEEE
  (2017)

\bibitem{nagarajan2018lego}
Nagarajan, H., Sundar, K., Hijazi, H., Bent, R.: Convex hull formulations for
  mixed-integer multilinear functions.
\newblock In: Proceedings of the XIV International Global Optimization Workshop
  (LEGO 18) (2018)

\bibitem{nagarajan2016optimal}
Nagarajan, H., Yamangil, E., Bent, R., Van~Hentenryck, P., Backhaus, S.:
  Optimal resilient transmission grid design.
\newblock In: Power Systems Computation Conference (PSCC), 2016, pp. 1--7. IEEE
  (2016)

\bibitem{puranik2017domain}
Puranik, Y., Sahinidis, N.V.: Domain reduction techniques for global {NLP} and
  {MINLP} optimization.
\newblock Constraints \textbf{22}(3), 338--376 (2017)

\bibitem{Rikun1997}
Rikun, A.D.: {A Convex Envelope Formula for Multilinear Functions}.
\newblock Journal of Global Optimization \textbf{10}, 425--437 (1997).
\newblock \doi{10.1023/A:1008217604285}

\bibitem{ruiz2017global}
Ruiz, J.P., Grossmann, I.E.: Global optimization of non-convex generalized
  disjunctive programs: a review on reformulations and relaxation techniques.
\newblock Journal of Global Optimization \textbf{67}(1-2), 43--58 (2017)

\bibitem{ryoo1995global}
Ryoo, H.S., Sahinidis, N.V.: Global optimization of nonconvex nlps and {MINLP}s
  with applications in process design.
\newblock Computers \& Chemical Engineering \textbf{19}(5), 551--566 (1995)

\bibitem{ryoo2001analysis}
Ryoo, H.S., Sahinidis, N.V.: Analysis of bounds for multilinear functions.
\newblock Journal of Global Optimization \textbf{19}(4), 403--424 (2001)

\bibitem{sahinidis1996baron}
Sahinidis, N.V.: Baron: A general purpose global optimization software package.
\newblock Journal of global optimization \textbf{8}(2), 201--205 (1996)

\bibitem{Speakman2017}
Speakman, E.E.: {Volumetric Guidance for Handling Triple Products in Spatial
  Branch-and-Bound by}.
\newblock Ph.D. thesis, University of Michigan (2017)

\bibitem{tawarmalani2005polyhedral}
Tawarmalani, M., Sahinidis, N.V.: A polyhedral branch-and-cut approach to
  global optimization.
\newblock Mathematical Programming \textbf{103}(2), 225--249 (2005)

\bibitem{teles2013univariate}
Teles, J.P., Castro, P.M., Matos, H.A.: Univariate parameterization for global
  optimization of mixed-integer polynomial problems.
\newblock European Journal of Operational Research \textbf{229}(3), 613--625
  (2013)

\bibitem{trespalacios2016cutting}
Trespalacios, F., Grossmann, I.E.: Cutting plane algorithm for convex
  generalized disjunctive programs.
\newblock INFORMS Journal on Computing \textbf{28}(2), 209--222 (2016)

\bibitem{Vielma2009}
Vielma, J.P., Nemhauser, G.L.: Modeling disjunctive constraints with a
  logarithmic number of binary variables and constraints.
\newblock Mathematical Programming \textbf{128}(1), 49--72 (2011)

\bibitem{wicaksono2008piecewise}
Wicaksono, D.S., Karimi, I.: Piecewise {MILP} under-and overestimators for
  global optimization of bilinear programs.
\newblock AIChE Journal \textbf{54}(4), 991--1008 (2008)

\bibitem{fei2017acc}
Wu, F., Nagarajan, H., Zlotnik, A., Sioshansi, R., Rudkevich, A.: Adaptive
  convex relaxations for gas pipeline network optimization.
\newblock In: American Control Conference (ACC), 2017, pp. 4710--4716. IEEE
  (2017)

\end{thebibliography}

\appendix
\section{Appendix}

\subsection{Sensitivity Analysis of $\Delta$}
One of the important details of MINLP algorithms and approaches is their parameterization. As seen in the earlier sections, AMP is no different. The quality of the solutions depend heavily on the choice of $\Delta$. However, in spite of this problem specific dependence, it is often interesting to identify reasonable default values. Table \ref{Tab:delta_range} presents computational results on all instances for different choices of $\Delta$. From these results, AMP is most effective when $\Delta$ is between 4 and 10.

\begin{table}[htp]
    \centering
    \caption{This table shows a sensitivity analysis of AMP's  performance to the choice of $\Delta$. Here, we bin results by $\Delta \le 4$, $\Delta$ between 4 and 10, and $\Delta > 10$. From these results, it is clear that most of the good choices of $\Delta$ are between 4 and 10 and this is our recommended choice for this parameter. For each instance, the bold face font represents best run time or the best optimality gap (if the solve times out)}
    \begin{tabular}{l|rr|rr|rr}
\toprule
& \multicolumn{2}{c}{$\Delta\leq 4$} & \multicolumn{2}{|c|}{$4<\Delta\leq 10$} & \multicolumn{2}{c}{$\Delta>10$} \\
\cmidrule(lr){2-3} \cmidrule(lr){4-5} \cmidrule(lr){6-7}
Instances & Gap(\%) & $T$ & Gap(\%) & $T$ & Gap(\%) & $T$ \\
\midrule
p1 & GOpt & 0.74 & GOpt & 0.24 & \textbf{GOpt} & \textbf{0.06} \\
p2 & GOpt & 0.60 & GOpt & 0.20 & \textbf{GOpt} & \textbf{0.10} \\
fuel & \textbf{GOpt} & \textbf{0.05} & GOpt & 0.06 & GOpt & 0.07 \\
ex1223a & GOpt & 0.03 & \textbf{GOpt} & \textbf{0.02} & GOpt & 0.02 \\
ex1264 & GOpt & 1.92 & \textbf{GOpt} & \textbf{0.79} & GOpt & 1.03 \\
ex1265 & GOpt & 2.24 & \textbf{GOpt} & \textbf{0.28} & GOpt & 0.80 \\
ex1266 & GOpt & 0.20 & GOpt & 0.23 & \textbf{GOpt} & \textbf{0.16} \\
eniplac & GOpt & 2.39 & GOpt & 1.46 & \textbf{GOpt} & \textbf{0.75} \\
util & GOpt & 2.52 & GOpt & 2.14 & \textbf{GOpt} & \textbf{0.55} \\
meanvarx & GOpt & 967.70 & GOpt & 118.80 & \textbf{GOpt} & \textbf{70.09} \\
blend029 & GOpt & 1.98 & GOpt & 1.33 & \textbf{GOpt} & \textbf{1.00} \\
blend531 & GOpt & 88.60 & GOpt & 74.33 & \textbf{GOpt} & \textbf{49.76} \\
blend146 & 23.34 & TO & \textbf{1.60} & \textbf{TO} & 3.20 & TO \\
blend718 & GOpt & 1263.41 & \textbf{GOpt} & \textbf{581.68} & GOpt & 889.72 \\
blend480 & 0.10 & TO & \textbf{0.02} & \textbf{TO} & 0.02 & TO \\
blend721 & GOpt & 486.17 & \textbf{GOpt} & \textbf{44.13} & GOpt & 176.11 \\
blend852 & 0.01 & TO & \textbf{GOpt} & \textbf{144.26} & GOpt & 322.80 \\
wtsM2\_05 & GOpt & 2236.45 & GOpt & 2545.41 & \textbf{GOpt} & \textbf{386.95} \\
wtsM2\_06 & 0.02 & TO & \textbf{GOpt} & \textbf{519.38} & GOpt & 972.20 \\
wtsM2\_07 & 0.77 & TO & 0.57 & TO & \textbf{0.54} & \textbf{TO} \\
wtsM2\_08 & \textbf{7.92} & \textbf{TO} & 9.28 & TO & 11.96 & TO \\
wtsM2\_09 & \textbf{7.47} & \textbf{TO} & 68.58 & TO & 68.58 & TO \\
wtsM2\_10 & 0.11 & TO & 0.11 & TO & \textbf{0.10} & \textbf{TO} \\
wtsM2\_11 & \textbf{6.10} & \textbf{TO} & 6.27 & TO & 10.41 & TO \\
wtsM2\_12 & 6.49 & TO & 8.69 & TO & \textbf{4.00} & \textbf{TO} \\
wtsM2\_13 & 7.37 & TO & \textbf{2.03} & \textbf{TO} & 10.27 & TO \\
wtsM2\_14 & 4.06 & TO & 5.59 & TO & \textbf{1.43} & \textbf{TO} \\
wtsM2\_15 & \textbf{0.17} & \textbf{TO} & 0.17 & TO & 0.57 & TO \\
wtsM2\_16 & 5.73 & TO & 8.17 & TO & \textbf{5.25} & \textbf{TO} \\
lee1 & GOpt & 73.19 & \textbf{GOpt} & \textbf{13.61} & 0.03 & TO \\
lee2 & 0.38 & TO & \textbf{0.02} & \textbf{TO} & 0.08 & TO \\
meyer4 & GOpt & 64.82 & \textbf{GOpt} & \textbf{5.33} & GOpt & 20.74 \\
meyer10 & GOpt & 684.63 & \textbf{GOpt} & \textbf{133.47} & 9.70 & TO \\
meyer15 & \textbf{0.10} & \textbf{TO} & 0.33 & TO & 0.15 & TO \\
\hline \ \
Summary & \multicolumn{2}{c}{7} & \multicolumn{2}{c}{14} & \multicolumn{2}{c}{14} \\
\bottomrule
    \end{tabular}
    \label{Tab:delta_range}
\end{table}

\subsection{Logarithmic and Linear Encoding of Partition Variables}

In section \ref{sec:problem}, the discussion on piecewise convex relaxations described formulations that encoded the partition variables with a linear number of variables and a logarithmic number of variables \cite{Vielma2009}.
Table \ref{Tab:log_compare} compares the performance of AMP using both formulations. Despite fewer variables in the logarithmic formulation, this encoding is only effective on a few problems, generally on problems that require a significant number of partitions. These results suggest that when the logarithmic encoding has nearly the same number of partition variables as the linear encoding, the linear encoding is more effective. 

\begin{table}[htp]
    \centering
    \caption{This table compares the logarithmic formulation of partition variables with the linear representation. Each column indicates the formulation with the fastest runtime for different choices of $\Delta$. The last column enumerates the number of times the logarithmic formulation is better.}
    \begin{tabular}{lrrrrrr}
\toprule
Instances & $F (\Delta=4)$ & $F (\Delta=8)$ & $F (\Delta=10)$ & $F (\Delta=16)$ & $F (\Delta=32)$ & Total \\
\midrule

eniplac & lin & lin & {log} & {log} & {log} & 3 \\
blend531 & lin & lin & {log} & lin & lin & 1 \\
blend146 & lin & lin & {log} & lin & lin & 1 \\
blend718 & lin & lin & lin & lin & {log} & 1 \\
blend480 & lin & lin & {log} & lin & lin & 1 \\
blend721 & lin & lin & lin & lin & lin & 0 \\
blend852 & {log} & lin & {log} & lin & {log} & 3 \\
wtsM2\_05 & lin & lin & lin & lin & lin & 0 \\
wtsM2\_06 & lin & lin & lin & lin & lin & 0 \\
wtsM2\_07 & lin & lin & lin & lin & lin & 0 \\
wtsM2\_08 & {log} & lin & {log} & lin & {log} & 3 \\
wtsM2\_09 & lin & {log} & {log} & {log} & {log} & 4 \\
wtsM2\_10 & lin & lin & lin & lin & lin & 0 \\
wtsM2\_11 & lin & lin & lin & lin & lin & 0 \\
wtsM2\_12 & lin & lin & lin & lin & lin & 0 \\
wtsM2\_13 & lin & lin & lin & lin & lin & 0 \\
wtsM2\_14 & lin & lin & lin & lin & lin & 0 \\
wtsM2\_15 & {log} & lin & lin & {log} & lin & 2 \\
wtsM2\_16 & lin & lin & lin & lin & lin & 0 \\
lee1 & {log} & {log} & {log} & {log} & {log} & 5 \\
lee2 & {log} & lin & lin & lin & lin & 1 \\
meyer4 & {log} & {log} & lin & lin & {log} & 3 \\
meyer10 & lin & lin & {log} & {log} & lin & 2 \\
meyer15 & {log} & {log} & lin & lin & {log} & 3 \\
\hline ${Total}$ & 7 & 4 & 9 & 5 & 8 & 33 \\
\bottomrule
    \end{tabular}
    \label{Tab:log_compare}
\end{table}

\end{document}